\documentclass[lettersize,journal]{IEEEtran}
\usepackage{cite}
\usepackage{amsmath,amssymb,amsfonts}
\usepackage{algorithm}
\usepackage{algcompatible}
\usepackage{graphicx}
\usepackage{textcomp}
\usepackage{float}
\usepackage{multicol}
\usepackage{arydshln}
\usepackage{amsmath}
\usepackage{amsthm}
\usepackage{amssymb}
\usepackage{amsfonts}
\usepackage{graphicx}
\usepackage{mathrsfs}
\usepackage{subfigure}
\usepackage{url}
\usepackage{booktabs}
\usepackage{float}
\usepackage{color}
\usepackage{bm}
\usepackage{graphicx}
\usepackage{epstopdf}
\usepackage{algpseudocode}
\usepackage{arydshln}

\newcommand{\col}{\hbox{col}}
\newtheorem{assmp}{\bf Assumption}

\newtheorem{pro}{\bf Problem}

\newtheorem{rem}{\bf Remark}

\newtheorem{lem}{\bf Lemma}

\newtheorem{thm}{\bf Theorem}

\newtheorem{Example}{\bf Example}

\newenvironment{Prf}{\noindent{\emph{Proof:}}}{\hfill $\Box$\par}

\newcommand{\EQ}{\begin{eqnarray}}
\newcommand{\EN}{\end{eqnarray}}
\newcommand{\EQQ}{\begin{eqnarray*}}
\newcommand{\ENN}{\end{eqnarray*}}

\ifCLASSINFOpdf
\else
\fi

\begin{document}
\title{Data-Driven Output Feedback based Analysis and Control for {\color{black}
Unknown Discrete-Time Linear System}}

\author{Haoyan~Lin~and~Jie~Huang,~\IEEEmembership{Life Fellow,~IEEE}
\thanks{This work was supported in part by the Research Grants Council of the Hong Kong Special Administrative Region under grant No. 14203924.}
\thanks{{\color{black}The authors are with the Department }of Mechanical and Automation Engineering, The Chinese University of Hong Kong, Hong Kong (e-mail:hylin@mae.cuhk.edu.hk; jhuang@mae.cuhk.edu.hk. Corresponding author: Jie Huang.)}}

\maketitle

\begin{abstract}
Using the notion of data informativity,
the existing results have given conditions under which the data are informative for controller designs for various control problems of unknown linear discrete-time systems, and have developed methods to compute such controllers from data. Nevertheless, the existing conditions for computing a dynamic output feedback control law based on the input and output data are somehow stringent. In this paper,  we first focus on
developing data informativity analysis and control for an unknown system with a known input matrix.
It turns out that the informativity conditions for such a system are much milder and the methods for computing state feedback control laws for stabilization, deadbeat control and the linear-quadratic regulator (LQR)
for such a system are much more straightforward. Further, based on the parameterized observer,
we  show that the problem of computing a dynamic output feedback control law for an unknown system can be converted to the problem of designing a state feedback control law for an ancillary system whose input
matrix is known. Therefore, the results of the first part of this paper can be directly used to compute a dynamic output feedback control law for stabilization and deadbeat control for unknown linear discrete-time systems based on the input and output data.  Moreover, we present a dynamic output feedback control law which will asymptotically approach {\color{black}a state feedback LQR solution} to the original unknown system.
\end{abstract}


\IEEEpeerreviewmaketitle

\section{Introduction}
 Data-driven control has received extensive attention over the past decades. It aims to design a control law using collected data to satisfy various control objectives. There are  two main classes of approaches in data-driven control: indirect data-driven methods and direct data-driven methods. Indirect data-driven methods need to identify the system model from the measured data first \cite{Ljung1987, Overschee1996}. Then, the model-based control techniques can be applied to find the control law. However, if the ultimate objective is the control policy, it might be easier to learn the control policy than the model \cite{Dörfler2023}. The direct data-driven methods aim to bypass the system identification and directly derive the controller from data. In what follows, we focus on the literature utilizing the direct data-driven methods.

For uncertain discrete-time linear systems, Q-learning is a popular method to solve the LQR problem, encompassing methods such as {\color{black}policy iteration (PI)} \cite{Bradtke1994} and {\color{black}value iteration (VI)} \cite{Landelius1997}.
Nevertheless, the conditions on data required in the above literature are more stringent than those for system identification.
This raises the question of whether the control problem can be solved when the required data conditions are no stricter than those for system identification. Motivated by the fundamental lemma presented in \cite{Willems2005}, {\color{black}the authors in \cite{Persis2019}} successfully transformed the stabilization problem and LQR problem for the unknown discrete-time linear system into solving the {\color{black}linear matrix inequality (LMI)} and convex optimization problem constructed with the collected data.
\cite{Persis2019} solved the stabilization problem and LQR problem for the unknown linear discrete-time system under the same data conditions  as those for system identification. However, as pointed out in \cite{Waarde2020}, the results in \cite{Persis2019} assume that the input is persistently exciting of sufficiently high order.
Therefore, \cite{Waarde2020} further focused on finding the minimal condition on data for verifying the control system property and solving the control problem. {\color{black}In this context, data are defined as informative if they enable the assessment of system properties or the direct synthesis of a control law for the physical system from the available data.}
An interesting discovery of \cite{Waarde2020} is that the data conditions for finding the stabilizing feedback gain are weaker than those for system identification. {\color{black}For further discussions on robustness and optimization problems related to unknown discrete-time systems, readers are refer to \cite{Waarde2023, Donge2024, Dörfler2022}.}

For uncertain continuous-time linear systems, integral reinforcement learning (IRL) is an effective technique to approximately solve the optimal control problem.
\cite{Vrabie2009} first studied the LQR problem for unknown linear systems with known input matrix using the {\color{black}PI approach} to iteratively solve {\color{black}a continuous algebraic Riccati equation (CARE)}. \cite{Jiang2012} further extended the results to completely unknown linear systems using a PI-based approach. Since the PI-based approach requires an initial stabilizing feedback gain to start the iteration, \cite{Bian2016} introduced a {\color{black}VI approach} to address the LQR problem for unknown linear systems without needing an initial stabilizing feedback gain.

Most of the literature focuses on the state feedback control law, which assumes
that the input and state data are available. However, for most real applications, the full-state measurement is not available. Thus, it is necessary to develop a methodology to synthesize output feedback control laws based on the input and output data only.
\cite{Lewis2010} considered solving the LQR problem for the unknown discrete-time linear system with the input and output data. Their approach introduced a discounting factor  in the cost function, which helps to diminish the effect of excitation noise bias, the resultant control law  in \cite{Lewis2010} is only suboptimal. To avoid using the discounting factor, \cite{Rizvi2022} developed a dynamic output feedback control law based on a parameterized  observer that is independent of the original system dynamics. Thus, the approach of \cite{Rizvi2022} can be used to design a dynamic output feedback control law that will asymptotically approach  a state feedback {\color{black}LQR solution.} It is {\color{black}noteworthy that \cite{Persis2019}} also studied the stabilization problem for the unknown linear discrete-time system based on input and output data. They converted the  output feedback stabilization problem of an unknown system to a state feedback stabilization problem for an ancillary system whose state comprised a stack of past input and output data from the original system.  However, as mentioned above, the results in \cite{Persis2019} assume that the input is persistently exciting of sufficiently high order. Moreover,
the paper did not take advantage of the fact that the input matrix of the ancillary system is known and their  results rely on more stringent conditions. {\color{black} Additionally, \cite{Waarde2020} developed observer-based dynamic controllers in its input/state/output framework.  Nonetheless, this work primarily focuses on reconstructing a stacked state and transforming the input/output setting into the input/state/output setting, and it only obtained sufficient conditions for the data to be informative for stabilization by dynamic measurement feedback.}

In the paper, we aim to further address the issue of designing a dynamic output feedback control law for such control problems as stabilization, deadbeat control and LQR for unknown discrete-time linear systems based on the  input and output data.
 The main contributions of this paper are summarized {\color{black}as follows:}
\begin{itemize}
\item [1)]
We first focus on developing data informativity analysis and control for an unknown system with a known input matrix.
It turns out that the  informativity conditions for such a system are much milder and the methods for computing state feedback control laws for  stabilization, deadbeat control, and the LQR
for such a system are much more straightforward. While the result of this part is interesting on its own
for the case where the input matrix is known as studied in \cite{Vrabie2009}, it lays a foundation for further studying the output feedback control of an unknown linear system.

\item [2)] Based on the given plant, we propose an unknown ancillary system whose input matrix is known.   We show that this system is stabilizable if the original system is stabilizable.
{\color{black}Further, we show that a state feedback control law that stabilizes this system will lead to a dynamic output feedback control law for the original system.}
Thus, we convert the problem of designing an output feedback control law for the original unknown system to the design of a state feedback control law for an ancillary system whose input
matrix is known. Therefore, the results of the first part of this paper can be directly used to compute a dynamic output feedback control law for such problems as stabilization and deadbeat control for unknown linear discrete-time systems based on the input and output data of the original unknown system.

\item [3)]
{\color{black}We show that by selecting the weights of the cost functions appropriately, the state feedback LQR solution for the ancillary system will asymptotically approach the state feedback LQR solution for the original unknown system. }
Besides, the gain for the state feedback LQR solution of the ancillary system can be obtained by solving an LMI and a convex optimization problem based on the input and output data of the original system under a milder condition on the data.  It is noted that \cite{Persis2019} and \cite{Waarde2020}  did not consider the relationship between the solutions to the LQR problem for the original system and the ancillary system.

\end{itemize}


The rest of the paper is organized as follows. Section \ref{Preliminaries} gives the preliminaries on the data informativity and {\color{black}formulates the problems.}
Sections \ref{data infor stabilization} and \ref{data infor LQR} present the main results on data informativity for stabilization by state feedback and for the LQR problem, respectively, in the case where the input and state data are available and the input matrix is known.
Section \ref{data infor io} considered a more general case of data informativity where only the input and output data are available. Section \ref{conclusion} concludes the paper.

\indent\textbf{Notation}
$\mathbb{R}$ represent the sets of real numbers.
$\mbox{col} (x_1,\cdots,x_N )=[x_1^T,\cdots,x_N^T]^T  \in \mathbb{R}^{Nn}$ with $x_i \in \mathbb{R}^n, i=1,\cdots,N$.
$I_{N}$ represents an $N \times N$ identity matrix. $\textup{tr}(A)$ represents the trace of a square matrix $A$.
{\color{black}
$\lambda(A)$ denotes the spectrum of $A$.}
For a positive semidefinite matrix $Q$, $\sqrt{Q}$ denotes the positive semidefinite square root.
A nilpotent matrix is a square matrix $A$ such that $A^{k}=\bf{0}$ for some positive integer $k$.

\section{Preliminaries and Problem Formulation}\label{Preliminaries}
{\color{black}
\indent In this section, we will summarize the main results on data informativity as presented in \cite{Waarde2020} and formulate the problems.}

\subsection{Data Informativity \cite{Waarde2020}} \label{sec2.1}
Let $\Sigma$ be the set of {\color{black}all discrete-time linear systems} of the form
{\color{black}\begin{subequations}  \label{lisys2}
	\begin{align}
		x(t+1)&=Ax(t)+Bu(t) \label{lisys}\\
		y(t)&=Cx(t)
	\end{align}
\end{subequations}}where $x(t) \in \mathbb{R}^n$ is the state, $u(t) \in \mathbb{R}^m$ is the input, and  $y(t) \in \mathbb{R}^p$ is the output.
Let the ``true'' system be represented by the triplet $(A_s, B_s, C_s)$ which is unknown.
{\color{black}In this subsection, we assume that the input and state data can be collected on $k$ time intervals $\{0,1,\cdots, T_i \}$ for $i=1,2,\cdots,k$.} Let
\begin{subequations}
\begin{align}
U^i_{-}=&\begin{bmatrix}
u^i(0) & u^i(1) & \cdots &  u^i(T_i-1)
\end{bmatrix}\\
X^i=&\begin{bmatrix}
x^i(0) & x^i(1) & \cdots &  x^i(T_i)
\end{bmatrix}
\end{align}
\end{subequations}
denote the input and the state data on the $i$th interval. Define
\begin{subequations}
\begin{align}
X^i_{-}=&\begin{bmatrix}
x^i(0) & x^i(1) & \cdots &  x^i(T_i-1)
\end{bmatrix}\\
X^i_{+}=&\begin{bmatrix}
x^i(1) & x^i(2) & \cdots &  x^i(T_i)
\end{bmatrix}
\end{align}
\end{subequations}
Then, we have $X^i_{+} = A_s X^i_{-} + B_s  U^i_{-} $. {\color{black}Also, let}
\begin{subequations}
\begin{align}
U_{-}=&\begin{bmatrix}
U^1_{-} & \cdots &  U^k_{-}
\end{bmatrix},
X=\begin{bmatrix}
X^1 &  \cdots &  X^k
\end{bmatrix}\\
X_{-}=&\begin{bmatrix}
X^1_{-} & \cdots &  X^k_{-}
\end{bmatrix},
X_{+}=\begin{bmatrix}
X^1_{+} & \cdots &  X^k_{+}
\end{bmatrix}
\end{align}
\end{subequations}
{\color{black}where $N=\sum_{i=1}^{k} T_i$, $X_{+}, X_{-} \in \mathbb{R}^{n \times N}$, $U_{-} \in \mathbb{R}^{m \times N}$. Then, the data set is defined as $\mathcal{D}=(U_{-},X)$.}

{\color{black}
Let
\begin{align*}
	\Sigma^{0}_{i/s} & =\{ (A_0, B_0)  | {\bf{0}} = \begin{bmatrix} A_0 & B_0 \end{bmatrix} \begin{bmatrix}X_{-} \\ U_{-} \end{bmatrix} \}  \\
	\Sigma_{i/s} & =\{(A , B ) | X_+ = \begin{bmatrix} A & B \end{bmatrix} \begin{bmatrix}X_{-} \\ U_{-} \end{bmatrix} \} \\
\Sigma_{K}&=\{(A, B ) |A+BK \text{ is stable}\} \\
\Sigma_{K}^{\text{nil}}&=\{(A, B ) |A+BK \text{ is nilpotent}\}
 \end{align*}}where {\color{black}$K \in  \mathbb{R}^{m \times n}$.} Then, the data $\mathcal{D}$ are said to be informative for system identification if $\Sigma_{i/s}=\{(A_s, B_s)\}$. The data $\mathcal{D}$ are informative for system identification if and only if \cite{Waarde2020}
\begin{equation}
	\textup{rank} \begin{bmatrix}
		X_{-}\\
		U_{-}	
	\end{bmatrix}=n+m.
\end{equation}
{\color{black}
The data $\mathcal{D}$ are said to be informative for stabilization by state feedback if there exists a $K \in  \mathbb{R}^{m \times n}$ such that
$\Sigma_{i/s} \subseteq \Sigma_{K}$.}
A state feedback  controller of the form $u(t)=Kx(t)$ that will
drive the state of the system  to the origin in a finite time is called a deadbeat controller.
A deadbeat controller has the property that $(A_s+B_sK)^tx_0=\bf{0}$ for all $t \geq n$ and all $x_0 \in \mathbb{R}^n$.   {\color{black}Thus, $u(t)=Kx(t)$} is a deadbeat controller if and only if $A_s+B_sK$ is nilpotent.
The data $\mathcal{D}$ are informative for deadbeat control if there exists a feedback gain $K \in  \mathbb{R}^{m \times n}$ such that $\Sigma_{i/s} \subseteq \Sigma_{K}^{\text{nil}}$.

The  necessary and sufficient conditions for  $\mathcal{D}$ to be  informative for stabilization by state feedback is as follows.

\begin{thm}(Theorem 16 of \cite{Waarde2020})\label{thm1}
	The data $\mathcal{D}$ are informative for stabilization by state feedback if and only if the matrix $X_{-}$ has full row rank and there exists a right inverse $X_{-}^{\dagger}$ of $X_{-}$ such that $X_{+}X_{-}^{\dagger}$ is stable.
\end{thm}

\begin{rem}
	It is noted that the condition on the data required in the stabilization problem by state feedback is weaker than that of system identification.
\end{rem}

However, Theorem \ref{thm1} does not tell how to find a control gain $K$. The following result addresses this issue.

\begin{thm}(Theorem 17 of \cite{Waarde2020}) \label{thm2}
The data $\mathcal{D}$ are informative for stabilization by state feedback if and only if there exists a matrix $\Theta \in \mathbb{R}^{N\times n}$ satisfying
\begin{equation} \label{LMIWaarde}
{\color{black}X_{-} \Theta= (X_{-} \Theta )^T} \textup{ and }
\begin{bmatrix}
X_{-} \Theta & X_{+} \Theta\\
\Theta^T X_{+}^T & X_{-} \Theta
\end{bmatrix} {\color{black}\succ 0.}
\end{equation}
Moreover, $K$ satisfies $\Sigma_{i/s} \subseteq \Sigma_{K}$ if and only if $K=U_{-} \Theta (X_{-} \Theta)^{-1}$.
\end{thm}

The necessary and sufficient conditions for  $\mathcal{D}$ to be informative for deadbeat control are as follows:

\begin{thm}(Theorem 21 of \cite{Waarde2020})\label{thmdeadbeat}
The data $\mathcal{D}$ are informative for deadbeat control if and only if the matrix $X_{-}$ has full row rank and there exists a right inverse $X_{-}^{\dagger}$ of $X_{-}$ such that $X_{+}X_{-}^{\dagger}$ is nilpotent.
\end{thm}

Next, we will introduce the LQR problem of system \eqref{lisys} defined as follows.  Given system \eqref{lisys}, for any initial condition {\color{black}$x_0 \in \mathbb{R}^n$}, find an input $u^*$ that minimizes the cost function
\begin{equation} J(x_0,u)=\sum_{t=0}^{\infty} x^T(t)Qx(t)+u^T(t)Ru(t)
\end{equation}
where {\color{black}$Q \succeq 0$ and $R \succ 0$}.

{\color{black}
\begin{thm}(Corollary 2.1 of \cite{Molinari1975} and Theorem 23 of \cite{Waarde2020}) \label{thm3}
Let $Q=Q^T\succeq 0$ and $R=R^T\succ 0$. The following statements hold.

1. If the pair $(A,B)$ is stabilizable, then there exists a unique largest real symmetric solution $P^{+}$ to \eqref{DARE} in the sense that $P^{+} \succeq P$ for every real symmetric $P$ satisfying \eqref{DARE}.  The matrix $P^{+}$ is positive semidefinite.

2. The pair $(A,B)$ is stabilizable and the pair $(A, Q)$ (or equivalently $(A, \sqrt{Q})$) is observable for every eigenvalue of $A$ on the unit circle if and only if the following discrete-time algebraic Riccati equation (DARE) admits {\color{black}a real symmetric positive semidefinite solution $P^*$}
\begin{equation}\label{DARE}
P=A^TPA-A^TPB(R+B^TPB)^{-1}B^TPA+Q
\end{equation}
and the optimal input is given by $u^*(t)=K^*x(t)$ with
\begin{equation} \label{Kstar}
K^*=-(R+B^TP^*B)^{-1}B^TP^*A
\end{equation}
where $K^*$ is such that $A+BK^*$ is stable.
\end{thm}
 Suppose the LQR problem is solvable for $(A, B, Q, R)$ by the optimal feedback gain $K^*$ given by \eqref{Kstar}.  Then,  for any given $K$, define
\begin{equation}
\Sigma_{K}^{Q,R}=\{(A, B) | K \text{ is the optimal gain for } (A, B, Q, R) \}.
\end{equation}
That is, $\Sigma_{K}^{Q,R}$ is  the set of all systems of the form \eqref{lisys} for which $K$ is the optimal feedback gain corresponding to $Q$ and $R$.

Given matrices $Q$ and $R$, the data $\mathcal{D}$ are informative for LQR if there exists $K$ such that $\Sigma_{i/s} \subseteq \Sigma_{K}^{Q,R}$.
Now we are ready to rephrase the main results of \cite{Waarde2020} on  data informativity for LQR.

\begin{thm}(Theorem 28 of \cite{Waarde2020}) \label{thm4}
Let $Q=Q^T$ be positive semidefinite and $R=R^T$ be positive definite. The data $\mathcal{D}$ are informative for LQR if and only if at least one of {\color{black}the following two conditions hold:}
\begin{itemize}
 \item [1.] The data $\mathcal{D}$ are informative for system identification. Equivalently, there exists a $\begin{bmatrix}
V_1 & V_2
\end{bmatrix}$ such that $\begin{bmatrix}
X_{-}\\
U_{-}
\end{bmatrix} \begin{bmatrix}
V_1 & V_2
\end{bmatrix}=\begin{bmatrix}
I_n & \bf{0}\\
\bf{0} & I_m
\end{bmatrix}$. Moreover, the LQR problem is solvable for $(A_s, B_s, Q, R)$, where $A_s=X_{+}V_1$ and $B_s=X_{+}V_2$.

\item [2.] {\color{black}There exists a} $\Theta \in \mathbb{R}^{N\times n}$ such that $X_{-} \Theta= (X_{-} \Theta )^T$, $U_{-} \Theta=\bf{0}$,
\begin{equation} \nonumber
	\begin{bmatrix}
		X_{-} \Theta & X_{+} \Theta\\
		\Theta^T X_{+}^T & X_{-} \Theta
	\end{bmatrix} \succ 0
\end{equation}
and $QX_{+} \Theta=\bf{0}$.
\end{itemize}
\end{thm}

\begin{thm}(Theorem 29 of \cite{Waarde2020})
Let $Q=Q^T$ be positive semidefinite and $R=R^T$ be positive definite. Suppose that the data $\mathcal{D}$ are informative for LQR. Consider the linear operator $P \mapsto \mathcal{L}(P)$ defined by
\begin{equation}
\mathcal{L}(P)=X_{-}^TPX_{-}-X_{+}^TPX_{+}-X_{-}^TQX_{-}-U_{-}^TRU_{-}.
\end{equation}
Let $P^{+}$ be the unique largest real symmetric solution to \eqref{DARE}. The following statements hold.
\begin{itemize}
\item The matrix $P^{+}$ is equal to the unique solution to the optimization problem
\begin{equation} \nonumber
\begin{aligned}
&\textup{max}~ \textup{tr}(P)\\
&\textup{subject to } P \succeq 0 \textup{ and } \mathcal{L}(P) \preceq  0.
\end{aligned}
\end{equation}

\item There exists a right inverse $X_{-}^{\dagger}$ of $X_{-}$ such that
\begin{equation} \label{LP}
\begin{aligned}
 \mathcal{L}(P^{+})X_{-}^{\dagger}=\bf{0}.
\end{aligned}
\end{equation}
Moreover, if $X_{-}^{\dagger}$ satisfies \eqref{LP}, then the optimal feedback gain is given by $K=U_{-}X_{-}^{\dagger}$.
\end{itemize}

\end{thm}}

\subsection{Problem Formulation}\label{sec2.2}
In many cases, the full state of the system is not available.
Thus, we need to further consider the problem of designing control laws based on the input and output measurements of  discrete-time systems of the form \eqref{lisys2}.

It has been shown in \cite{Lin2024} that the data-driven output feedback control problem for continuous-time systems can be converted to the data-driven state feedback control problem of a well-defined ancillary system with the input matrix known. In Section V, we will further show that the data-driven output feedback control problem for discrete-time systems of the form \eqref{lisys2} can also be converted to the data-driven state feedback control problem of  a well-defined ancillary system of the form \eqref{lisys2} with the input matrix  known. Thus, in this paper, we will first sharpen the results as summarized in Section \ref{sec2.1} for the case where the input matrix $B$ is known. Then we will address in Section V  the data-driven output feedback control problem for the system  \eqref{lisys2}.

Since $B_s$ is known, we introduce the following sets:
{\color{black}
\begin{equation}\label{sigmais}
	\begin{aligned}
	\hat{\Sigma}^{0}_{i/s}& =\{ A^0  | {\bf{0}} = A^0 X_{-} \} \\
		\hat{\Sigma}_{i/s} &=\{ A | X_{+}-B_sU_{-}=AX_{-} \} \\
		\hat{\Sigma}_{K} &=\{A |A+B_s K \text{ is stable}\} \\
			\hat{\Sigma}_{K}^{\text{nil}}&=\{A |A+B_s K \text{ is nilpotent}\}
	\end{aligned}
\end{equation}}

Similar to the definitions in \cite{Waarde2020}, we say that the data $(U_{-},X)$ are informative for system identification if $\hat{\Sigma}_{i/s}=\{A_s\}$, which implies $\hat{\Sigma}^{0}_{i/s}=\{\bf{0}\}$,
the data $(U_{-},X)$ are informative for stabilization by state feedback  if there exists {\color{black}a state feedback control gain $K$} such that $\hat{\Sigma}_{i/s} \subseteq \hat{\Sigma}_{K}$, and
the data $\mathcal{D}$ are informative for deadbeat control if there exists a feedback gain $K$ such that $\hat{\Sigma}_{i/s} \subseteq \hat{\Sigma}_{K}^{\text{nil}}$.

Thus, the first three problems to be addressed in this paper are as follows:

\begin{pro} \label{p1}
Assuming $B_s$ is known, find the necessary and sufficient conditions for $\mathcal{D}$ to be informative for stabilization by state feedback.
\end{pro}

\begin{pro} \label{p2}
	Assuming $B_s$ is known, and  the data $\mathcal{D}$ are informative for stabilization by state feedback, find a control law $\mathcal{K}$ such that $	\hat{\Sigma}_{i/s} \subseteq 	\hat{\Sigma}_{K} $.
\end{pro}

\begin{pro} \label{p2x}
	Assuming $B_s$ is known, and  the data $\mathcal{D}$ are informative for deadbeat control, find a control law $\mathcal{K}$ such that $	\hat{\Sigma}_{i/s} \subseteq 	\hat{\Sigma}^{nil}_{K} $.
\end{pro}

We will also consider the data informativity problem for LQR with the input matrix known. For this purpose, we need to consider the following semidefinite programming problem:

\begin{pro}\label{op1} Given  $Q=Q^T \succeq 0$ and $R=R^T \succ 0$,  find $P$ such that
	\begin{equation} \nonumber
		\begin{aligned}
			&\textup{max } \textup{tr}(P)\\
			&\textup{subject to } P \succeq 0 \\
			&\begin{bmatrix}
				\bar{A}_s^T P \bar{A}_s -X_{-}^T P X_{-} + X_{-}^TQ X_{-} & \bar{A}_s^T P B_s \\
				B_s ^T P \bar{A}_s & R+B_s^T P B_s
			\end{bmatrix}  \succeq 0
		\end{aligned}
	\end{equation}
	where $\bar{A}_s=A_sX_{-} =X_{+}-B_s U_{-}$ and $B_s$ is known.
\end{pro}

Additionally, in   Section V,  we will  consider various design problems based on the input and output data of system \eqref{lisys2} with the input matrix known.

\section{Data Informativity for Stabilization by State Feedback with Input Matrix Known}\label{data infor stabilization}

Let us first note that it is evident that the data $(U_{-},X)$ are informative for system identification when $B_s$ is known if and only if
\begin{equation}\label{c14}
	\begin{aligned}
		\textup{rank}(X_{-})=n.
	\end{aligned}
\end{equation}
{\color{black}
\begin{rem}
The rank condition \eqref{c14} is much more milder than the rank condition $\text{rank}\left( \begin{bmatrix}X_{-} \\ U_{-}\end{bmatrix}\right)= n+m$ for  the general case where $B_s$ is unknown since the rank condition \eqref{c14} imposes no requirement on the rank of the matrix $U_{-}$, indicating  that only the state trajectories need to be sufficiently rich for identification. The “gap” between these two conditions reflects the fact that prior knowledge of $B_s$ eliminates the need to excite and identify the input data, thereby relaxing the data requirements. This weakened condition, in turn, weakens the condition for the data $(U_{-},X)$ to be informative for stabilization by state feedback, as will be made clear in Lemma \ref{lem1}.
\end{rem}}

We then establish the following lemma.

\begin{lem}\label{lem1}
	Suppose that the data $(U_{-},X)$ are informative for stabilization by state feedback, {\color{black}and let $K$ be a given feedback gain such that $	\hat{\Sigma}_{i/s} \subseteq 	\hat{\Sigma}_{K}$. Then $	\hat{\Sigma}^{0}_{i/s}=\{\bf{0}\}$ and hence  $X_{-}$ has full row rank.}
\end{lem}

\begin{Prf}
The proof is similar to the proof of Lemma 15 of \cite{Waarde2020}.
	
First, we prove that $A^0$ is nilpotent for all $A^0 \in \hat{\Sigma}^{0}_{i/s}$. Let $A^0 \in \hat{\Sigma}^{0}_{i/s}$, $A \in \hat{\Sigma}_{i/s}$, and $F=A+B_s K$. Then we can show that $A+\alpha A^0 \in \hat{\Sigma}_{i/s}, \forall \alpha \in \mathbb{R}$.
	
By the assumption in Lemma \ref{lem1}, $A+B_s K$ is stable for all $A \in \hat{\Sigma}_{i/s}$. Thus, the matrix $F+\alpha A^0$ is stable, which implies the spectral radius of the matrix $F+\alpha A^0$ is smaller than $1$. Let $\alpha > 0$ and divide $F+\alpha A^0$ by $\alpha$, the spectral radius of the matrix $\frac{1}{\alpha} F+A^0$ is smaller than $\frac{1}{\alpha}$. Consider the case where $\alpha$ tends to infinity, we can conclude that $A^0$ is nilpotent.
	
Moreover, $(A^0)^TA^0 \in \hat{\Sigma}^{0}_{i/s}$ if  $A^0 \in \hat{\Sigma}^{0}_{i/s}$. Thus, $(A^0)^TA^0$ is nilpotent.  Using the fact that the only symmetric nilpotent matrix is the zero matrix, we can conclude that $A^0=\bf{0}$, $\forall A^0 \in \hat{\Sigma}^{0}_{i/s}$, which implies that $X_{-}$ has full row rank.
\end{Prf}
{\color{black}
\begin{rem}
It is noted that Lemma 15 of \cite{Waarde2020} handled the similar problem for the general case where  $B_s$ is unknown and obtain the necessary condition for $A_0 + B_0 K = 0$ for all $(A_0, B_0) \in \Sigma^0_{i/s}$, which is equivalent to the following  $$\textup{im}\begin{bmatrix}
	I_n\\
	K
\end{bmatrix} \subseteq \textup{im}\begin{bmatrix}
	X_{-}\\
	U_{-}
\end{bmatrix}.$$
In contrast, Lemma  \ref{lem1} here imposes no constraint on $K$. As a result,  Lemma  \ref{lem1} applies to some systems that cannot be handled by Lemma 15 of \cite{Waarde2020}. 
\end{rem}}

With the above Lemma, we are ready to give the sufficient and necessary condition for the data to be informative for the stabilization.

{\color{black}
\begin{lem}\label{lemStateInfo}  
The data $(U_{-},X)$ are informative for stabilization by state feedback if and only if the matrix $X_{-}$ has full row rank and there exists a right inverse $X_{-}^{\dagger}$ of $X_{-}$ and $T \in \mathbb{R}^{m \times n}$ such that $X_{+}X_{-}^{\dagger}+B_sT$ is stable.
\end{lem}}

\begin{Prf}
If part: {\color{black}If the matrix $X_{-}$ has full row rank, then $ \hat{\Sigma}_{i/s}=\{A_s\}$.}
Let $K=U_{-}X_{-}^{\dagger}+T$. Then
\begin{equation}\label{eqK}
\begin{aligned}
X_{+}X_{-}^{\dagger}+B_s T=&\begin{bmatrix}A_s &  B_s \end{bmatrix}\begin{bmatrix} X_{-} \\  U_{-}\end{bmatrix}X_{-}^{\dagger}+B_s T\\
=&A_s+B_s K
\end{aligned}
\end{equation}
which is stable.
{\color{black}Therefore, there exists a state feedback control gain $K=U_{-}X_{-}^{\dagger}+T$ such that $\hat{\Sigma}_{i/s} \subseteq \hat{\Sigma}_{K}$ implies that the data $(U_{-},X)$ are informative for stabilization by state feedback.}
	
Only if part:
By Lemma \ref{lem1}, we have $A^0=\bf{0}$ which implies $X_{-}$ has full row rank and $ \hat{\Sigma}_{i/s}=\{A_s\}$.
And there exists a $K$ such that $A+B_s K$ is stable for all $A \in  \hat{\Sigma}_{i/s}$, which implies $A_s+B_s K$ is stable.
Let $T=K-U_{-} X_{-}^{\dagger}$ where $X_{-}^{\dagger}$ is any matrix such that $X_{-} X_{-}^{\dagger}=I_n$. Then,
\begin{equation}
\begin{aligned}
&A_s+B_s K\\
=&A_s+B_s (T+U_{-}X_{-}^{\dagger})\\
=&A_s+B_s U_{-} X_{-}^{\dagger} +B_s T\\
=&X_{+}X_{-}^{\dagger}+B_s T
\end{aligned}
\end{equation}
Thus, $X_{+}X_{-}^{\dagger}+B_s T$ is stable.
\end{Prf}

\begin{rem}
{\color{black} It is interesting to compare the condition given in Lemma \ref{lemStateInfo} and the condition given in Theorem \ref{thm1}. Theorem \ref{thm1} requires that the matrix  $X_{+}X_{-}^{\dagger}$ be Schur while Lemma \ref{lemStateInfo}  only requires  
the pair $(X_{+}X_{-}^{\dagger}, B_s)$ be stabilizable. Thus, 	
 we have one more variable $T$  to make the closed-loop system stable.
The following example shows that $X_{+}X_{-}^{\dagger}$ is not stable for all  $X_{-}^{\dagger}$, but there exists a $T$ such that $X_{+}X_{-}^{\dagger}+B_sT$ is stable for some $X_{-}^{\dagger}$.}
\end{rem}

\begin{Example}\label{eg1}
Consider system \eqref{lisys} with
\begin{equation} \label{examplesys}
\begin{aligned}
			A_s=\begin{bmatrix}
				1 & 0\\
				0 & 0
			\end{bmatrix},
			B_s=\begin{bmatrix}
				1\\
				1
			\end{bmatrix}.
\end{aligned}
\end{equation}
We collect data from this system on a single time interval from $t = 0$ until $t = 3$ with $u(t)=x_1(t)$ where $x_1(t)$ is the first component of $x(t)$. The collected data are  as follows:
	\begin{equation}\label{collecteddata}
		\begin{aligned}
			X=\begin{bmatrix}
				1   &  2 &    4  &   8\\
				2   &  1  &  2   &  4
			\end{bmatrix},
			U_{-}=\begin{bmatrix}
				1 &   2  & 4
			\end{bmatrix}.
		\end{aligned}
	\end{equation}

Clearly, $X_{-}$ is of full rank, but $\begin{bmatrix}
	X_{-}\\
	U_{-}
\end{bmatrix}$ is not since $U_{-}$ is linearly dependent of $X_{-}$.
Thus, the data $(U_{-},X_{-})$ are not informative for system identification if $A_s$ and $B_s$ are unknown.
The general solution of $X_{-}^{\dagger}$  is
\begin{equation}
\begin{aligned}
X_{-}^{\dagger}=\begin{bmatrix}
-\frac{1}{3}  &  \frac{2}{3} \\
\frac{2}{3}-2a & -\frac{1}{3}-2b  \\
a & b
\end{bmatrix}
\end{aligned}
\end{equation}
where $a, b \in \mathbb{R}$. For all $ X_{-}^{\dagger}$, we have
\begin{equation}
\begin{aligned}
X_{+}X_{-}^{\dagger}=\begin{bmatrix}
2 & 0\\
1 & 0
\end{bmatrix}
\end{aligned}
\end{equation}
Thus $X_{+}X_{-}^{\dagger}$ is not stable for all $X_{-}^{\dagger}$. Therefore, the method given in Theorem \ref{thm1} fails to deal with this case.
	
Nevertheless, by Lemma \ref{lemStateInfo}, taking advantage of the fact that $B_s$ is known and letting
\begin{equation}
\begin{aligned}
T=\begin{bmatrix}
-1.5000 &  0
\end{bmatrix}
\end{aligned}
\end{equation}
shows that $X_{+}X_{-}^{\dagger}+B_sT = \begin{bmatrix}
0.5 & 0\\
0 & 0
\end{bmatrix}$ is stable.
\end{Example}

Next, we further show how to find the stabilizing feedback gain in the following result.

\begin{thm}\label{thm7}
The data $(U_{-},X)$ are informative for stabilization by state feedback if and only if there exists matrices $\Theta \in \mathbb{R}^{N \times n}$ and $T_p \in \mathbb{R}^{m \times n}$ such that
\begin{small}
\begin{equation}\label{LMI1}
\begin{aligned}
{\color{black}X_{-} \Theta= (X_{-} \Theta )^T} \textup{ and }
\begin{bmatrix}
X_{-} \Theta & X_{+}\Theta+B_s T_p\\
(X_{+}\Theta+B_s T_p)^T & X_{-} \Theta
\end{bmatrix}\succ 0.
\end{aligned}
\end{equation}
\end{small}{\color{black}Let $T=T_p (X_{-} \Theta)^{-1}$, then, $K=U_{-} \Theta (X_{-} \Theta)^{-1}+T$ satisfies $\hat{\Sigma}_{i/s} \subseteq \hat{\Sigma}_{K}$.}
\end{thm}

\begin{Prf}
If part: Since \eqref{LMI1} is satisfied, $X_{-} \Theta \succ 0$ which implies $X_{-}$ has full row rank.
Applying Schur complement and letting $T=T_p (X_{-} \Theta)^{-1}$ into \eqref{LMI1} gives 
{\color{black}
\begin{scriptsize}
\begin{equation} \label{eq24}
\begin{aligned}
&X_{-} \Theta-(X_{+}\Theta+B_s T_p)(X_{-} \Theta)^{-1}(X_{+}\Theta+B_s T_p)^T\\
=&X_{-} \Theta-(X_{+}\Theta+B_s T_p)(X_{-} \Theta)^{-1}(X_{-} \Theta) (X_{-} \Theta)^{-1} (X_{+}\Theta+B_s T_p)^T\\
=&X_{-} \Theta-(X_{+}\Theta(X_{-} \Theta)^{-1}+B_s T)(X_{-} \Theta)(X_{+}\Theta(X_{-} \Theta)^{-1}+B_s T)^T\\
\succ &0
\end{aligned}
\end{equation}
\end{scriptsize}}
Since \eqref{eq24} is a DARE with positive definite solution $X_{-} \Theta \succ 0$,  {\color{black}$X_{+}\Theta(X_{-} \Theta)^{-1}+B_s T$ is stable.} With $X_{-}^{\dagger}=\Theta(X_{-} \Theta)^{-1}$, $X_{+}X_{-}^{\dagger}+B_s T$ is stable.
By Lemma \ref{lemStateInfo}, the data $(U_{-},X)$ are informative for stabilization by state feedback. {\color{black}Moreover, by \eqref{eqK}, there exists a state feedback gain $K=U_{-} X_{-}^{\dagger}+T$ such that $\hat{\Sigma}_{i/s} \subseteq \hat{\Sigma}_{K}$.}
	
Only if part:
By Lemma \ref{lemStateInfo}, if the data $(U_{-},X)$ are informative for stabilization by state feedback, then the matrix $X_{-}$ has full row rank and there exist a right inverse $X_{-}^{\dagger}$ of $X_{-}$ and $T \in \mathbb{R}^{m \times n}$ such that $X_{+}X_{-}^{\dagger}+B_sT$ is stable.
	
Since $X_{+}X_{-}^{\dagger}+B_sT$ is stable, $(X_{+}X_{-}^{\dagger}+B_sT)^T$ is also stable. Then, there exists a $P \succ 0$ such that
\begin{equation}
\begin{aligned}
(X_{+}X_{-}^{\dagger}+B_sT)P(X_{+}X_{-}^{\dagger}+B_sT)^T-P<0
\end{aligned}
\end{equation}
Let $X_{-}^{\dagger} P=\Theta$ and $T_p=TP$, we have $P=X_{-} \Theta$. Then, the above equation can be rewritten as
\begin{equation}
\begin{aligned}
&(X_{+}X_{-}^{\dagger}+B_sT)P(X_{+}X_{-}^{\dagger}+B_sT)^T-P\\
=&(X_{+}X_{-}^{\dagger}P+B_sTP)P^{-1} (X_{+}X_{-}^{\dagger}P+B_sTP)^T-P\\
=&(X_{+}\Theta+B_sT_p)(X_{-} \Theta)^{-1} (X_{+}\Theta+B_sT_p)^T-(X_{-} \Theta)<0
\end{aligned}
\end{equation}
which implies \eqref{LMI1}.
\end{Prf}

{\color{black}
\begin{rem} 
Table \ref{tab} shows the assumptions, required rank conditions, and the informative conditions given in \cite{Persis2019}, \cite{Waarde2020}, and this paper.
	\begin{table*}
	\caption{Comparison of Assumptions, Rank Conditions, and Informative Conditions}\label{tab}
	\center
	{\color{black}
	\begin{tabular}{|c|c|c|c|}\hline
		Reference & Assumptions & Rank Conditions  & Informative Conditions\\ 
		\hline
		\cite{Persis2019} & $A_s,B_s$ are unknown & $\text{rank}\left(\begin{bmatrix} X_{-} \\  U_{-}\end{bmatrix} \right)= n+m$ & $\begin{bmatrix}
			X_{-} \Theta & X_{+} \Theta\\
			\Theta^T X_{+}^T & X_{-} \Theta
		\end{bmatrix} \succ 0$ \\
		\hline
		\cite{Waarde2020} &$A_s, B_s$ are unknown  &   $\text{rank}\left(X_{-}\right)= n$ &  $\begin{bmatrix}
			X_{-} \Theta & X_{+} \Theta\\
			\Theta^T X_{+}^T & X_{-} \Theta
		\end{bmatrix} \succ 0$ \\
		\hline  
		This paper & $A_s$ is unknown, $B_s$ is known &   $\text{rank}\left(X_{-}\right)= n$ & $\begin{bmatrix}
			X_{-} \Theta & X_{+}\Theta+B_s T_p\\
			(X_{+}\Theta+B_s T_p)^T & X_{-} \Theta
		\end{bmatrix}\succ 0$ \\
		\hline  		
	\end{tabular}}
\end{table*}
Comparing Theorem \ref{thm7} with Theorem \ref{thm2}, \eqref{LMI1} has an advantage of having an additional matrix $T_p$ to fulfill the positive definite  condition. This advantage  is illustrated by the following example which shows that \eqref{LMIWaarde} cannot be satisfied for all $\Theta$,   but there exists a $T_p$ such that \eqref{LMI1} is satisfied for some $\Theta$.
\end{rem}}

\begin{Example}\label{ex2}
Consider system \eqref{examplesys} again. Let
\begin{equation}
\begin{aligned}
\Theta=\begin{bmatrix}
x_1 & x_2\\
x_3 & x_4\\
x_5 & x_6
\end{bmatrix}
\end{aligned}
\end{equation}
where $x_1, x_2, \cdots, x_6$ are some variables to be determined.

Then, \eqref{eq27} holds.
\begin{figure*}[ht]
{\color{black}
\begin{equation} \label{eq27}
\begin{aligned}
&\begin{bmatrix}
X_{-} \Theta & X_{+} \Theta\\
\Theta^T X_{+}^T & X_{-} \Theta
\end{bmatrix}\\
=&\begin{bmatrix}
 x_1 + 2x_3 + 4x_5 & x_2 + 2x_4 + 4x_6 & 2x_1 + 4x_3 + 8x_5 & 2x_2 + 4x_4 + 8x_6\\
 2x_1 + x_3 + 2x_5 & 2x_2 + x_4 + 2x_6 & x_1 + 2x_3 + 4x_5 & x_2 + 2x_4 + 4x_6\\
2x_1 + 4x_3 + 8x_5 & x_1 + 2x_3 + 4x_5 & x_1 + 2x_3 + 4x_5 & x_2 + 2x_4 + 4x_6\\
2x_2 + 4x_4 + 8x_6 & x_2 + 2x_4 + 4x_6 & 2x_1 + x_3 + 2x_5 &  2x_2 + x_4 + 2x_6
\end{bmatrix}
\end{aligned}
\end{equation}}
\end{figure*}
Since $X_{-} \Theta= (X_{-} \Theta )^T$, $x_2 + 2x_4 + 4x_6=2x_1 + x_3 + 2x_5$.
To verify whether $\begin{bmatrix}
	X_{-} \Theta & X_{+} \Theta\\
	\Theta^T X_{+}^T & X_{-} \Theta
\end{bmatrix}$ is positive definite or not, we apply Sylvester's Criterion which requires that all the leading principal minors of it are positive. Thus, we have
\begin{equation} \label{pminor}
\begin{aligned}
 Q_1=&x_1 + 2x_3 + 4x_5>0\\
 Q_2=&3x_2x_3 - 3x_1x_4 - 6x_1x_6 + 6x_2x_5>0\\
 Q_3=&3Q_1(2x_1x_3 - 2x_1x_2 - x_1x_4 - 3x_2x_3 + 4x_1x_5 \nonumber\\
 &- 2x_1x_6 - 6x_2x_5 + x_1^2)>0\\
 Q_4=&3Q_2 (2x_1x_3 - 2x_1x_2 - x_1x_4 - 3x_2x_3 + 4x_1x_5 \nonumber\\
 &- 2x_1x_6 - 6x_2x_5 + x_1^2)>0,
\end{aligned}
\end{equation}
which can be simplified as
\begin{subequations}
\begin{align}
&x_1 + 2x_3 + 4x_5>0\\
&3x_2x_3 - 3x_1x_4 - 6x_1x_6 + 6x_2x_5>0 \label{ieq2}\\
&2x_1x_3 - 2x_1x_2 - x_1x_4 - 3x_2x_3 + 4x_1x_5- 2x_1x_6 \nonumber \\
&- 6x_2x_5 + x_1^2>0 \label{ieq3}
\end{align}
\end{subequations}

Substituting $x_2=2x_1 + x_3 + 2x_5-2x_4 - 4x_6$ into \eqref{ieq2} and \eqref{ieq3} gives
\begin{subequations}
\begin{align}
&6x_1x_3 - 3x_1x_4 + 12x_1x_5 - 6x_1x_6 - 6x_3x_4 + 12x_3x_5 \nonumber \\
& -12x_3x_6 - 12x_4x_5 - 24x_5x_6 + 3x_3^2 + 12x_5^2 >0 \label{ieq4}\\
&3x_1x_4 - 6x_1x_3 - 12x_1x_5 + 6x_1x_6 + 6x_3x_4 - 12x_3x_5 \nonumber \\
& + 12x_3x_6 + 12x_4x_5 + 24x_5x_6 - 3x_1^2 - 3x_3^2 - 12x_5^2>0 \label{ieq5}
\end{align}
\end{subequations}

Adding \eqref{ieq4} to \eqref{ieq5} gives $-3x_1^2>0$, which forms a contradiction.
Thus, \eqref{LMIWaarde} cannot be satisfied for all $\Theta$.
However, taking advantage of the fact that $B_s$ is known, we can find
\begin{equation}
\begin{aligned}
\Theta=\begin{bmatrix}
-0.3522  &  1.0565\\
8.9914  & -213.9669\\
-4.1435 & 106.7193
\end{bmatrix},
T_p=\begin{bmatrix}
-1.5848  &  0
\end{bmatrix}
\end{aligned}
\end{equation}
such that \eqref{LMI1} is satisfied.
\end{Example}

 {\color{black} 
 \begin{rem}
 The LMIs in both the approach of this paper and the approaches in \cite{Persis2019}, \cite{Waarde2020}
 have the same size. That is, the dimension of the LMIs in all three papers is $2n$ where $n$ is the dimension of the system. Thus, the computational complexities in all three papers are essentially the same. Nevertheless, since we have one more matrix $T_p \in \mathbb{R}^{m \times n}$ at our disposal,  our LMI involves a total of $(N+m)n$ variables while the LMIs in \cite{Persis2019}, \cite{Waarde2020} only involve $Nn$ variables. 
 \end{rem}}

The necessary and sufficient conditions for  $\mathcal{D}$ to be informative for deadbeat control is as follows.

\begin{thm}\label{thmdeadbeat2}
	The data $(U_{-},X)$ are informative for deadbeat control if and only if the matrix $X_{-}$ has full row rank and there exists a right inverse $X_{-}^{\dagger}$ of $X_{-}$ and $T \in \mathbb{R}^{m \times n}$ such that $X_{+}X_{-}^{\dagger}+B_sT$ is nilpotent.
\end{thm}

\begin{Prf}
If part:  Since $X_{-}$ has full row rank,  $ \hat{\Sigma}_{i/s}=\{A_s\}$.
Let $K=U_{-}X_{-}^{\dagger}+T$. Then
\begin{equation}
\begin{aligned}
X_{+}X_{-}^{\dagger}+B_s T=&\begin{bmatrix}A_s &  B_s \end{bmatrix}\begin{bmatrix} X_{-} \\  U_{-}\end{bmatrix}X_{-}^{\dagger}+B_s T\\
			=&A_s+B_s K
\end{aligned}
\end{equation}
which is nilpotent.
Therefore, the data $(U_{-},X)$ are informative for deadbeat control.
	
Only if part:
Since the data $(U_{-},X)$ are informative for deadbeat control,  the data $(U_{-},X)$ are informative for stabilization by state feedback.
By Lemma \ref{lem1}, we have $A^0=\bf{0}$, which implies $X_{-}$ has full row rank and $ \hat{\Sigma}_{i/s}=\{A_s\}$.
And there exists a $K$ such that $A+B_s K$ is nilpotent for all $A \in  \hat{\Sigma}_{i/s}$, which implies $A_s+B_s K$ is nilpotent.
Let $T=K-U_{-} X_{-}^{\dagger}$ where $X_{-}^{\dagger}$ is any matrix such that $X_{-} X_{-}^{\dagger}=I_n$. Then,
	\begin{equation}
		\begin{aligned}
			&A_s+B_s K\\
			=&A_s+B_s (T+U_{-}X_{-}^{\dagger})\\
			=&A_s+B_s U_{-} X_{-}^{\dagger} +B_s T\\
			=&X_{+}X_{-}^{\dagger}+B_s T
		\end{aligned}
	\end{equation}
Thus, $X_{+}X_{-}^{\dagger}+B_s T$ is nilpotent.
\end{Prf}

\begin{rem}
A full row rank  matrix $X_{-}$ may have multiple  right inverse $X_{-}^{\dagger}$ of $X_{-}$, {\color{black}but Theorem \ref{thmdeadbeat} does not tell how to find  a right inverse $X_{-}^{\dagger}$ of $X_{-}$ such that $X_{+}X_{-}^{\dagger}$ is nilpotent.} In contrast, with an extra freedom $T$, when $B_s$ is known,
it is straightforward to find a right inverse $X_{-}^{\dagger}$ of $X_{-}$ and a matrix $T \in \mathbb{R}^{m \times n}$ such that $X_{+}X_{-}^{\dagger}+B_sT$ is nilpotent.
	In fact,  suppose there exists a $X_{-}^{\dagger_1}$ and a $T_1$ such that $X_{+}X_{-}^{\dagger_1}+B_sT_1$ is nilpotent.  Then, for any right inverse $X_{-}^{\dagger_2}$  of $X_{-}$, let
$T_2=U_{-}X_{-}^{\dagger_1}+T_1-U_{-}X_{-}^{\dagger_2}$. Then
\begin{equation}
\begin{aligned}
&X_{+}X_{-}^{\dagger_2}+B_sT_2\\
=&\begin{bmatrix}A_s &  B_s \end{bmatrix}\begin{bmatrix} X_{-} \\  U_{-}\end{bmatrix}X_{-}^{\dagger_2}+B_s T_2\\
=&A_s+B_s U_{-} X_{-}^{\dagger_2} +B_s (U_{-}X_{-}^{\dagger_1}+T_1-U_{-}X_{-}^{\dagger_2})\\
=&A_s+B_s (U_{-}X_{-}^{\dagger_1}+T_1)\\
=&(A_sX_{-}+B_s U_{-})X_{-}^{\dagger_1}+B_s T_1\\
=&X_{+}X_{-}^{\dagger_1}+B_sT_1
\end{aligned}
\end{equation}
which means $X_{+}X_{-}^{\dagger_2}+B_sT_2$ is nilpotent.

This fact suggests an easy way to find a right inverse $X_{-}^{\dagger}$ of $X_{-}$ and a right matrix $T$ as follows. First, find an arbitrary right inverse $X_{-}^{\dagger}$ of $X_{-}$.
{\color{black}Second, verify whether the pair $(X_{+}X_{-}^{\dagger}, B_s)$ is stabilizable and all the eigenvalues of $X_{+}X_{-}^{\dagger}$ not at the origin are controllable. If yes, one can apply the pole placement technique to find a $T$ such that all the eigenvalues of  $X_{+}X_{-}^{\dagger}+B_sT$ are at the origin. Otherwise, the data $(U_{-},X)$ are not informative for deadbeat control.}
\end{rem}

{\color{black}
\begin{Example}
Consider system \eqref{examplesys} again. Since it has been shown in Example \ref{eg1} that for all $X_{-}^{\dagger}$, we have
\begin{equation}
\begin{aligned}
X_{+}X_{-}^{\dagger}=\begin{bmatrix}
2 & 0\\
1 & 0
\end{bmatrix}
\end{aligned}
\end{equation}
which is not nilpotent. Therefore, the method given in Theorem \ref{thmdeadbeat} fails to deal with this case. Nevertheless, by Theorem \ref{thmdeadbeat2}, taking advantage of the fact that $B_s$ is known and letting
\begin{equation}
\begin{aligned}
T=\begin{bmatrix}
-2 &  0
\end{bmatrix}
\end{aligned}
\end{equation}
shows that $X_{+}X_{-}^{\dagger}+B_sT = \begin{bmatrix}
	0 & 0\\
  -1 & 0
\end{bmatrix}$ is nilpotent.
\end{Example}}

\section{Data Informativity for LQR with the  Known Input Matrix}\label{data infor LQR}
In this section, we will consider the data informativity for LQR problem with the input matrix known.
Given $Q=Q^T \succeq 0$  and $R=R^T \succ 0$, let
\begin{equation}
	\Sigma_{K}^{Q,R}=\{A | K \text{ is the optimal gain for } (A, B_s, Q, R) \},
\end{equation}
that is,  $\Sigma_{K}^{Q,R}$ is the set of all system matrices $(A, B_s)$ for which $K$ is the optimal feedback gain corresponding to $Q$ and $R$.

The following result gives the sufficient and necessary conditions for $\mathcal{D}$ to be informative for  the LQR problem.

\begin{thm}\label{thm8}
Let $Q=Q^T \succeq 0$ be such that $(A_s, \sqrt{Q})$ is detectable and $R=R^T$ be positive definite. Then the data $(U_{-},X)$ are informative for LQR if and only if the matrix $X_{-}$ has full row rank and there exist a right inverse $X_{-}^{\dagger}$ of $X_{-}$ and $T \in \mathbb{R}^{m \times n}$ such that $X_{+}X_{-}^{\dagger}+B_sT$ is stable.
\end{thm}
\begin{Prf}
If part: By Lemma \ref{lemStateInfo}, if the matrix $X_{-}$ has full row rank and there exists a right inverse $X_{-}^{\dagger}$ of $X_{-}$ and $T \in \mathbb{R}^{m \times n}$ such that $X_{+}X_{-}^{\dagger}+B_s T$ is stable, then $\Sigma_{i/s}=A_s$ and the pair $(A_s, B_s)$ is stabilizable. With $(A_s, \sqrt{Q})$ detectable and $R \succ 0$, the LQR problem is solvable for $(A_s, B_s, Q, R)$. Therefore, the data $(U_{-},X)$ are informative for LQR.
	
Only if part: If the data $(U_{-},X)$ are informative for LQR, the data $(U_{-},X)$ are informative for stabilization by state feedback. By Lemma \ref{lemStateInfo}, the matrix $X_{-}$ has full row rank and there exists a right inverse $X_{-}^{\dagger}$ of $X_{-}$ and $T \in \mathbb{R}^{m \times n}$ such that $X_{+}X_{-}^{\dagger}+B_sT$ is stable.
\end{Prf}

\begin{rem}
Different from Theorem \ref{thm4}, we assume that $(A_s, \sqrt{Q})$ is detectable which is slightly stricter than the condition given in Theorem \ref{thm3}. Nevertheless, under the assumption that the pair $(A_s, C_s)$ is observable, if  we choose $Q = C_s^T Q_y C_s$ with $Q_y = Q^T_y \succ 0$,
then   the pair $(A_s, \sqrt{Q})$ is detectable.
\end{rem}

Since Theorem \ref{thm8} does  not tell how to find the optimal feedback gain, we further
present the following result.

\begin{thm}\label{thm9}
Let $Q=Q^T \succeq 0$ be such that $(A_s, \sqrt{Q})$ is detectable and $R=R^T$ be positive definite. Suppose the data $(U_{-},X)$ are informative for LQR. Then the unique positive semidefinite solution $P^*$ to \eqref{DARE} is equal to the unique solution to Problem \ref{op1}. Moreover,
let $P^*$ be the optimal solution to Problem \ref{op1}. Then the optimal feedback gain $K^*=U_{-}X_{-}^{\dagger}+T^*$, where $X_{-}^{\dagger}$ is any matrix such that $X_{-} X_{-}^{\dagger}=I_n$, $T^*=-(R+B_s^T P^* B_s)^{-1} (B_s^T P^* X_{+} + R U_{-} ) X_{-}^{\dagger}$.
\end{thm}

\begin{Prf}
	Since the data $(U_{-},X)$ are informative for LQR,  $\hat{\Sigma}_{i/s}=A_s$ and the pair $(A_s,B_s)$ is stabilizable. Moreover, with $(A_s, \sqrt{Q})$ being detectable, the following DARE admits a unique positive semidefinite solution $P_s^*$ by Theorem 8 of \cite{kucera1972}:
	\begin{equation}\label{DAREs}
		P=A_s^TPA_s-A_s^TPB_s(R+B_s^TPB_s)^{-1}B_s^TPA_s+Q
	\end{equation}

	By Theorem 3.1 of \cite{Ran1988}, the unique positive semidefinite solution $P_s^*$  of  \eqref{DAREs} is the unique solution of the following optimization problem:

	\begin{equation} \label{op2}
		\begin{aligned}
			\textup{max } & \textup{tr}(P)\\
			\textup{subject to } &P=P^T \\
			&R+B_s^TP B_s \succ 0 \\
			& A_s^T P A_s-P +Q\\
			&-A_s^T P B_s(R+B_s^T P B_s)^{-1}B_s^T P A_s \succeq 0
		\end{aligned}
	\end{equation}

	Since $P_s^*$ is positive semidefinite, we can impose a constraint $P \succeq 0$ on $P$. Thus, with $R\succ 0$, $R+B_s^TP B_s\succ 0$ is satisfied automatically. Then, $P_s^*$ of  \eqref{DAREs} is the unique solution of the following optimization problem:
	
	\begin{equation}\label{op3}
		\begin{aligned}
			\textup{max } & \textup{tr}(P)\\
			\textup{subject to } &P \succeq 0 \\
			& A_s^T P A_s-P +Q\\
			&-A_s^T P B_s(R+B_s^T P B_s)^{-1}B_s^T P A_s \succeq 0
		\end{aligned}
	\end{equation}
	
	By  Schur complement, the constraint $A_s^T P A_s-P+Q-A_s^T P B_s(R+B_s^T P B_s)^{-1}B_s^TP A_s \succeq 0$ is equivalent to
	\begin{equation}\label{schur1}
		\begin{aligned}
			\begin{bmatrix}
				A_s^T P A_s-P+Q & A_s^T P B_s \\
				B_s^TP A_s & R+B_s^T P B_s
			\end{bmatrix} \succeq 0
		\end{aligned}
	\end{equation}
	
	Since the data $(U_{-},X)$ are informative for LQR,  $X_{-} $ has full row rank, so does $\begin{bmatrix}
		X_{-} & \bf{0} \\
		\bf{0} & I_m
	\end{bmatrix}$.
	
	Then, \eqref{schur1} is equivalent to
	\begin{small}
		\begin{equation}\label{schur2}
			\begin{aligned}
				&\begin{bmatrix}
					X_{-}^T & \bf{0} \\
					\bf{0} & I_m
				\end{bmatrix}
				\begin{bmatrix}
					A_s^T P A_s-P+Q & A_s^T P B_s \\
					B_s^TP A_s & R+B_s^T P B_s
				\end{bmatrix} \begin{bmatrix}
					X_{-} & \bf{0} \\
					\bf{0} & I_m
				\end{bmatrix}\\
				=& \begin{bmatrix}
					X_{-}^TA_s^T P A_sX_{-} -X_{-}^T P X_{-} + X_{-}^TQ X_{-} & X_{-}^T A_s^T P B_s\\
					B_s^T P A_s X_{-} & R+B_s^T P B_s
				\end{bmatrix}\\
				=& \begin{bmatrix}
					\bar{A}_s^T P \bar{A}_s -X_{-}^T P X_{-} + X_{-}^TQ X_{-} & \bar{A}_s^T P B_s \\
					B_s ^T P \bar{A}_s & R+B_s^T P B_s
				\end{bmatrix}  \succeq 0
			\end{aligned}
		\end{equation}
	\end{small}
	where $\bar{A}_s=A_sX_{-} =X_{+}-B_s U_{-}$ is known.
	
	Having shown that the optimization problem \eqref{op3} is equivalent to the optimization problem \ref{op1},  we will show how to find the optimal feedback gain $K^*$.
	
	Note that $K^*=-(R+B_s^T P^*B_s)^{-1}B_s^T P^*A_s$ is the unique solution of the following equation:
	\begin{equation}\label{LPK}
		\begin{aligned}
			\begin{bmatrix}
				B_s^T P^* A_s & R+B_s^T P^* B_s
			\end{bmatrix}\begin{bmatrix}I_n \\ K^*  \end{bmatrix}=\bf{0}
		\end{aligned}
	\end{equation}
	since $R+B_s^TP^* B_s \succ 0$.
	
	Let $K^*=U_{-} X_{-}^{\dagger}+T^*$, where $X_{-}^{\dagger}$ is any matrix such that $X_{-}X_{-}^{\dagger}=I_n$. For any $X_{-}^{\dagger}$,  $T^*$ is the unique solution of the following equation with fixed $X_{-}^{\dagger}$.
	\begin{equation}\label{LPT1}
		\begin{aligned}
			\begin{bmatrix}
				B_s^T P^* A_s & R+B_s^T P^* B_s
			\end{bmatrix}
			(\begin{bmatrix}X_{-} \\ U_{-}  \end{bmatrix}X_{-}^{\dagger}+\begin{bmatrix}\bf{0} \\ T^*  \end{bmatrix})=\bf{0}
		\end{aligned}
	\end{equation}
	
	The above equation implies
	\begin{equation}\label{LPT2}
		\begin{aligned}
			T^*=-(R+B_s^T P^* B_s)^{-1} (B_s^T P^* X_{+} + R U_{-} ) X_{-}^{\dagger}
		\end{aligned}
	\end{equation}
\end{Prf}

\begin{rem}
Different from \cite{Waarde2020}, we do not impose any restriction on the feedback gain $K$ while \cite{Waarde2020} required that $\textup{im}\begin{bmatrix}
I_n\\
K
\end{bmatrix} \subseteq \textup{im}\begin{bmatrix}
X_{-}\\
U_{-}
\end{bmatrix}$. Thus, the proof of Theorem \ref{thm9} is totally different from the proof of Theorem 29 in \cite{Waarde2020}.	
\end{rem}

{\color{black}
\begin{Example}
Recall that when $A_s$ and $B_s$ are unknown, the collected data given in \eqref{collecteddata} are not informative for system identification.
Moreover, as shown in  Example \ref{ex2},  \eqref{LMIWaarde} cannot be satisfied for all $\Theta$. Thus,  by Theorem 28 of \cite{Waarde2020}, for any $Q \succeq 0$ and $R \succ 0$, the data $(U_{-}, X)$ are not informative for LQR. 

Nevertheless, let $Q=I_2$ and $R=1$. Then, taking advantage of the fact that $B_s$ is known,  there exists a $X_{-}^{\dagger}$ and a $T$ such that $X_{+}X_{-}^{\dagger}+B_sT$ is stable. Thus, the data $(U_{-}, X)$ are informative for LQR when $B_s$ is known by Theorem \ref{thm8}. Solving Problem \ref{op1} gives $P=\begin{bmatrix}
	2.0000  &  -0.0000\\
	-0.0000  &  1.0000
\end{bmatrix}$.
Let $X_{-}^{\dagger}=\begin{bmatrix}
	-0.3333  &   0.6667\\
	0.1333 &  -0.0667\\
	0.2667 &  -0.1333
\end{bmatrix}$. Then, $T=-(R+B_s^T P B_s)^{-1} (B_s^T P X_{+} + R U_{-} ) X_{-}^{\dagger}=[-1.5000, 0.0000]$, $
K=U_{-}X_{-}^{\dagger}+T=[-0.5000, 0.0000]$.
Solving \eqref{DARE} gives $P^*=\begin{bmatrix}
	2 & 0\\
	0 & 1
\end{bmatrix}$, $K^*=\begin{bmatrix}
	-0.5000 & 0.0000
\end{bmatrix}$.
Thus,  the method proposed in Theorem \ref{thm9} works well.
\end{Example}}

\section{Data Informativity Using Input and Output Data}\label{data infor io}

In this section, we further focus on the data informativity problem using input and output data only. Suppose we can collect input and output data on $k$ time intervals $\{0,1,\cdots, T_i \}$ for $i=1,2,\cdots,k$.
Similar to  \cite{Lin2024}, we first define an ancillary system using the input and output data. Then, we will show that the solutions to the stabilization problem and LQR problem of the ancillary system will lead to the output-based solutions to the same problems of the original system.

\subsection{State Parameterization Technique \cite{Rizvi2022}}

\begin{assmp}\label{ass1}
	The pair $(A,C)$ is observable.
\end{assmp}

Under Assumption \ref{ass1}, there exists an observer gain $L$ such that the eigenvalues of $A-LC$ can be arbitrarily placed inside the unit circle. Then the following Luenberger observer:
\begin{align}\label{observer}
	\hat{x}(t+1)= &A\hat{x}(t)+Bu(t)-L(C\hat{x}(t)-y(t))\notag \\
	=&(A-LC)\hat{x}(t)+Bu(t)+Ly(t)
\end{align}
will drive the state $\hat{x}(t)$ converge to $x(t)$.

Let $\Lambda(z)=\textup{det}(zI-A+LC) = z^n + \alpha_{n-1} z^{n-1} + \cdots+ \alpha_1 z + \alpha_0$. Then $(zI-A+LC)^{-1}=\frac{D_{n-1}z^{n-1}+D_{n-2}z^{n-2}+\cdots+D_1z+D_0}{\Lambda(z) }$ where $D_i\in \mathbb{R}^{n\times n}, i=0,1,\cdots, n-1$. Let $\zeta^i(t) \in \mathbb{R}^n$ be governed  by the following system:
\begin{align} \label{observer2}
	\zeta^i(t+1)=\mathcal{A} \zeta^i(t)+b w_i(t), ~~i = 1, \cdots, m+p
\end{align}
where $w_i(t) = u_i(t) $, $i = 1, \cdots, m$,  $w_i(t) = y_{i-m}(t) $, $i = m+1, \cdots, m+p$, and
\begin{align*}
	\mathcal{A}=\begin{bmatrix}
		0&1&0& \cdots& 0\\
		0& 0& 1& \cdots & 0\\
		\vdots & \vdots&\vdots& \vdots& \vdots\\
		0 & 0& 0& \cdots& 1\\
		-\alpha_0 & -\alpha_1& -\alpha_2& \cdots& -\alpha_{n-1}
	\end{bmatrix},~ b=\begin{bmatrix}
		0\\0\\ \vdots \\0\\ 1
	\end{bmatrix}
\end{align*}

Let $\zeta(t)= \col (\zeta^1(t), \cdots,  \zeta^{p+m}(t)) \in \mathbb{R}^{n_\zeta}$ and  $M = [M_1,\cdots ,M_{p+m}]$ where $n_\zeta=n(m+p)$, $M_i =  [D_0 f_i ,\cdots ,D_{n-1} f_i]$ with $f_i$ the $i^{th}$ column of $B$ for $i =1, \cdots, m$ and $(i - m)^{th}$ column of $L$ for $i = m+1, \cdots, m+p$.

By \eqref{observer2} and the definition of $\zeta(t)$, $\zeta(t)$ is governed by the following system:
\begin{align}\label{dyz}
	\zeta(t+1)= (I_{m+p}\otimes \mathcal{A})
	\zeta(t)+\begin{bmatrix}
		I_m\otimes b\\ {\bf0}
	\end{bmatrix}u(t)+\begin{bmatrix}
		{\bf0}\\I_p\otimes b
	\end{bmatrix}y(t)
\end{align}
By Lemma 5 of \cite{Lin2024}, the following equalities hold
\begin{subequations} \label{rela0}
	\begin{align}
	M (I_{m+p}\otimes \mathcal{A}) =&(A-LC)M  \label{rela1} \\
	M\begin{bmatrix}
		I_m\otimes b\\ {\bf0}
	\end{bmatrix}=&B \label{rela2} \\
	M\begin{bmatrix}
		{\bf0}\\I_p\otimes b
	\end{bmatrix}=&L \label{rela3}
\end{align}
\end{subequations}

Let $e_x(t)\triangleq M\zeta(t)-x(t)$. Then,  using \eqref{rela0} gives
\begin{equation}\label{ex}
	\begin{aligned}
		&e_x(t+1)=M\zeta(t+1)-x(t+1)\\
		=&M\left((I_{m+p}\otimes \mathcal{A})
		\zeta(t)+\begin{bmatrix}
			I_m\otimes b\\ {\bf0}
		\end{bmatrix}u(t)+\begin{bmatrix}
			{\bf0}\\I_p\otimes b
		\end{bmatrix}y(t) \right)\\
		&-Ax(t)-Bu(t)\\
		=&(A-LC)M\zeta(t)+Bu(t)+Ly(t)-Ax(t)-Bu(t)\\
		=&(A-LC)e_x(t)
	\end{aligned}
\end{equation}

Since $A-LC$ is stable, $e_x(t)$ will decay to $\bf{0}$ exponentially. Thus, $M\zeta(t)$ tends to $x(t)$ exponentially. For this reason, we call \eqref{dyz} a parameterized observer for \eqref{lisys2}.

\begin{rem}
	Since $\Lambda(z)$ is user-defined, $\mathcal{A}$ and $b$ are known matrices. {\color{black} Since the eigenvalues of the system matrix $\mathcal{A}$ can be arbitrarily specified by adjusting the parameters $\alpha_0, \cdots, \alpha_{n-1}$ and $\mathcal{A}$ is stable, it is robust with respect to the variations of the parameters $\alpha_0, \cdots, \alpha_{n-1}$ and bounded measurement noise. }
	\end{rem}

\subsection{A Dynamic Output Feedback Controller}

Substituting $y(t)=Cx(t)=CM\zeta(t)-Ce_x(t)$ into   \eqref{dyz} gives
\begin{align}\label{dyz2bf}
\zeta(t+1)=A_\zeta \zeta(t)+B_\zeta u(t) +d_\zeta(t)
\end{align}
where
\begin{align*}
	A_\zeta&=
	(I_{m+p}\otimes \mathcal{A})+\begin{bmatrix}
		{\bf0}\\I_p\otimes b
	\end{bmatrix}CM\\
	B_\zeta&=\begin{bmatrix}
		I_m\otimes b\\ {\bf0}
	\end{bmatrix},d_\zeta(t)=-\begin{bmatrix}
		{\bf0}\\I_p\otimes b
	\end{bmatrix}Ce_x(t)
\end{align*}

It is noted that $A_\zeta$ is an unknown matrix since $CM$ is unknown, but  $B_\zeta$ is known.

{\color{black}Before we show that the solution to the state-based stabilization problem of \eqref{dyz2bf} will lead to the solution to the output-based stabilization problem of \eqref{lisys2}, we need the following Lemma. }

\begin{lem} \label{lem2}
	Under Assumption \ref{ass1}, the pair $(A_\zeta, B_\zeta)$ is stabilizable if and only if the pair $(A,B)$ is stabilizable.
\end{lem}
{\color{black}
\begin{Prf}
See the Appendix.
\end{Prf}}

\begin{rem}\label{rem9}
	The proof of the sufficient part of this Lemma is parallel to  Lemma 6 of \cite{Lin2024}, which considered  the continuous-time case. But  Lemma 6 of \cite{Lin2024} did not consider the necessary part.
\end{rem}

{\color{black}Now, we are ready to show that the solution to the state-based stabilization problem of \eqref{dyz2bf} will lead to the solution to the  output-based stabilization problem of \eqref{lisys2}. }

\begin{thm} \label{thm10}
Under Assumption 1, there exists a dynamic output feedback control law of the following form:
\begin{equation}\label{outputc}
\begin{aligned}
u(t)=&K_\zeta \zeta(t)\\
\zeta(t+1)=&(I_{m+p}\otimes \mathcal{A})
\zeta(t)+\begin{bmatrix}
	I_m\otimes b\\ {\bf0}
\end{bmatrix}u(t)+\begin{bmatrix}
	{\bf0}\\I_p\otimes b
\end{bmatrix}y(t)
\end{aligned}
\end{equation}
such that the closed-loop system is stable if and only if the pair $(A, B)$ is stabilizable.
\end{thm}
{\color{black}
\begin{Prf}
See the Appendix.
\end{Prf}}

\begin{rem}
As the only design parameter of the output feedback control law \eqref{outputc} is the matrix $K_\zeta$, Theorem \ref{thm10} implies that the solution to the state feedback stabilization problem of the following system
\begin{align}\label{dyz2bf2}
\zeta(t+1)=A_\zeta \zeta(t)+B_\zeta u(t),
\end{align}
which is obtained from \eqref{dyz2bf} by ignoring the term $d_\zeta(t)$, will lead to the  solution to the output feedback stabilization problem of system \eqref{lisys2}.
We call \eqref{dyz2bf2} an ancillary system. {\color{black}Later on, we will establish the relationship of the deadbeat control problem and LQR problem between system \eqref{lisys2} and the ancillary system \eqref{dyz2bf2}.}
 \end{rem}

{\color{black}
\begin{rem}
The robustness of the Luenburger observer to the measurement noise or data outliers has been well studied  in \cite{Waarde2020S} and \cite{Bisoffi2022}. 
Define the effect of the measurement noise and the data outliers as $\omega(t)$.
In the presence of measurement noise or data outliers, the observer dynamics \eqref{dyz2bf2} can be rewritten as follows:
\begin{equation} \label{re50}
	\begin{aligned}
		\zeta(t+1)=A_\zeta \zeta(t)+B_\zeta u(t) +d(t)
	\end{aligned}
\end{equation}
where the disturbance term $d(t)$ includes the effects of current and past measurement noise and data outliers. If the effect of
the measurement noise and the data outliers $\omega(t)$ is bounded, then $d(t)$ has bounded energy. Then, the robust stabilization of the system \eqref{re50} can be achieved using the same methods developed in \cite{Waarde2020S} and \cite{Bisoffi2022}.
\end{rem}}

{\color{black}
We now will establish the relationship between the deadbeat control problem of \eqref{lisys2} by state feedback control and by output feedback control of the form \eqref{outputc}.

\begin{thm}  \label{thmdeadbeat3}
Under Assumption 1, system \eqref{lisys2} admits an output feedback deadbeat control law of the form \eqref{outputc}  if and only if system \eqref{lisys2} admits a state feedback deadbeat control law $u(t)=Kx(t)$.
\end{thm}
{\color{black}
\begin{Prf}
See the Appendix.
\end{Prf}}

\subsection{{\color{black}Informative Problem and Control Design Problem for \eqref{dyz2bf2}}}\label{ImAL}
 Note that the state and input data of the ancillary system \eqref{dyz2bf2} is available. To eliminate the effect of $e_x(t)$, we will collect the data of $\zeta(t)$ after $t\geq T_0$ for some integer $T_0$ on $k$ time intervals $\{T_0,T_0+1,\cdots, T_i \}$ for $i=1,2,\cdots,k$ with $T_i >T_0$. Let
\begin{subequations}
	\begin{align}
		\bar{U}^i_{-}=&\begin{bmatrix}
			u^i(T_0) & u^i(T_0+1) & \cdots &  u^i(T_i-1)
		\end{bmatrix}\\
		\bar{X}^i_{-}=&\begin{bmatrix}
			\zeta^i(T_0) & \zeta^i(T_0+1) & \cdots &  \zeta^i(T_i-1)
		\end{bmatrix}\\
		\bar{X}^i_{+}=&\begin{bmatrix}
			\zeta^i(T_0+1) & \zeta^i(T_0+2) & \cdots &  \zeta^i(T_i)
		\end{bmatrix}\\
		\bar{X}^i=&\begin{bmatrix}
			\zeta^i(T_0) & \zeta^i(T_0+1) & \cdots &  \zeta^i(T_i)
		\end{bmatrix} \\
		\bar{Y}^i_{-}=&\begin{bmatrix}
			y^i(T_0) & y^i(T_0+1) & \cdots &  y^i(T_i-1)
		\end{bmatrix}
	\end{align}
\end{subequations}
Then, putting the above data together gives
\begin{subequations}
	\begin{align}
		\bar{U}_{-}=&\begin{bmatrix}
			\bar{U}^1_{-} & \cdots &  \bar{U}^k_{-}
		\end{bmatrix},
		\bar{X}=\begin{bmatrix}
			\bar{X}^1 &  \cdots &  \bar{X}^k
		\end{bmatrix}\\
		\bar{X}_{-}=&\begin{bmatrix}
			\bar{X}^1_{-} & \cdots &  \bar{X}^k_{-}
		\end{bmatrix},
		\bar{X}_{+}=\begin{bmatrix}
			\bar{X}^1_{+} & \cdots &  \bar{X}^k_{+}
		\end{bmatrix} \\
		\bar{Y}_{-}=&\begin{bmatrix}
			\bar{Y}^1_{-} & \cdots &  \bar{Y}^k_{-}
		\end{bmatrix}
	\end{align}
\end{subequations}

Let $\bar{T}=\sum_{i=1}^{k} (T_i-T_0)$. $\bar{X}_{+}, \bar{X}_{-} \in \mathbb{R}^{n_\zeta \times \bar{T}}$, $\bar{U}_{-} \in \mathbb{R}^{m \times \bar{T}}$.
Then, $\bar{X}_{+}=A_\zeta \bar{X}_{-}+B_\zeta \bar{U}_{-}$. {\color{black}Let $\bar{\mathcal{D}}=(\bar{U}_{-}, \bar{X})$.}

Since $B_\zeta$ is known, corresponding to \eqref{sigmais},   we introduce the following sets for system \eqref{dyz2bf2}:
{\color{black}
\begin{equation}\label{sigmaisz}
	\begin{aligned}
			\bar{\Sigma}^0_{i/s} &=\{ A^0_\zeta  | {\bf{0}} = A^0_\zeta \bar{X}_{-} \} \\
		\bar{\Sigma}_{i/s} &=\{ A_\zeta | \bar{X}_{+}-B_\zeta \bar{U}_{-}= A_\zeta \bar{X}_{-} \} \\
	\bar{\Sigma}_{K_\zeta} & =\{ A_\zeta | A_\zeta+B_\zeta K_\zeta \text{ is stable}\}  \\
		\bar{\Sigma}_{K_\zeta} ^{\text{nil}}& =\{ A_\zeta | A_\zeta+B_\zeta K_\zeta \text{ is nilpotent}\}  		
		\end{aligned}
\end{equation}}

Similar to Section \ref{sec2.2}, we say that the  data $\bar{\mathcal{D}}$ are informative for system identification of \eqref{dyz2bf2} if $\bar{\Sigma}_{i/s} =\{A_{\zeta,s}\}$ with $A_{\zeta,s}$ the true system matrix of \eqref{dyz2bf2}, which implies $\bar{\Sigma}^{0}_{i/s}=\{\bf{0}\}$, {\color{black}
the data $\bar{\mathcal{D}}$  are informative for stabilization of \eqref{dyz2bf2}  by state feedback  if there exists a feedback gain {\color{black}${K}_\zeta$ such that $ \bar{\Sigma}_{i/s} \subseteq \bar{\Sigma}_{K_\zeta}$, and 	the data $\bar{\mathcal{D}}$ are informative for deadbeat control of \eqref{dyz2bf2} by state feedback if there exists a feedback gain ${K}_\zeta$ such that $ \bar{\Sigma}_{i/s} \subseteq \bar{\Sigma}_{K_\zeta}^{\text{nil}}$.}

\begin{lem}
The data $\bar{\mathcal{D}}$ are informative for system identification of \eqref{dyz2bf2} if and only if
		\begin{equation}
		\begin{aligned}
			\textup{rank}(\bar{X}_{-})=n_\zeta
		\end{aligned}
	\end{equation}
	where $n_\zeta=n(m+p)$.
\end{lem}

 {\color{black}
\begin{rem}
\cite{Persis2019} also studied the stabilization problem for the unknown linear discrete-time system based on input and output data. They constructed an ancillary system as follows:
\begin{equation}\label{anci}
\begin{aligned}
\chi(t+1)=A_\chi \chi(t) +B_\chi u(t)
\end{aligned}
\end{equation}
where
$\chi(t)=\col(y(t-n), y(t-n+1), \cdots, y(t-1), u(t-n), u(t-n+1), \cdots, u(t-1)) \in \mathbb{R}^{n_\zeta}$.
\begin{tiny}
\begin{equation}\nonumber
\begin{aligned}
A_\chi=&\begin{bmatrix}
{\bf 0} & I_p &  {\bf 0} & \cdots & {\bf 0} & {\bf 0} & {\bf 0}  & {\bf 0} & \cdots & {\bf 0} \\

{\bf 0} & {\bf 0} & I_p & \cdots  & {\bf 0} & {\bf 0} & {\bf 0}  & {\bf 0} & \cdots & {\bf 0}\\

\vdots & \vdots & \vdots & \ddots & \vdots & \vdots & \vdots & \vdots & \ddots & \vdots\\

{\bf 0} & {\bf 0} & {\bf 0} & \cdots  & I_p & {\bf 0} & {\bf 0}  & {\bf 0} & \cdots & {\bf 0}\\

-A_1 & -A_2 & -A_3 & \cdots & -A_n & B_1 & B_2 & B_3 & \cdots & B_n\\

\hline

{\bf 0} & {\bf 0} &  {\bf 0} & \cdots & {\bf 0} & {\bf 0} & I_m & {\bf 0} & \cdots & {\bf 0} \\

{\bf 0} & {\bf 0} &  {\bf 0} & \cdots & {\bf 0} & {\bf 0} & {\bf 0} &  I_m & \cdots & {\bf 0} \\

\vdots & \vdots & \vdots & \ddots & \vdots & \vdots & \vdots & \vdots & \ddots & \vdots\\

{\bf 0} & {\bf 0} &  {\bf 0} & \cdots & {\bf 0} & {\bf 0} & {\bf 0} &  {\bf 0} & \cdots &  I_m \\

{\bf 0} & {\bf 0} &  {\bf 0} & \cdots & {\bf 0} & {\bf 0} & {\bf 0} &  {\bf 0} & \cdots &  {\bf 0} \\
\end{bmatrix},\\
B_\chi=&\begin{bmatrix}
{\bf 0} & {\bf 0} & \cdots &  {\bf 0} & {\bf 0} & {\bf 0} & {\bf 0} & \cdots &  {\bf 0} & I_m
\end{bmatrix}^T
\end{aligned}
\end{equation}
\end{tiny}
and  $A_1, A_2, \cdots, A_n, B_1, B_2, \cdots, B_n$ are some matrices determined by $A, B, C$ and thus are unknown.
Then the paper showed
the  output feedback stabilization problem of  \eqref{lisys2} can be converted to the state feedback stabilization problem for the ancillary system (\ref{anci}). However, as pointed out in \cite{Waarde2020}, the results in \cite{Persis2019} assume that the input is persistently exciting of sufficiently high order.
Moreover, the paper did not take advantage of the fact that the input matrix of the ancillary system is known. To obtain an output feedback control law, one needs to
verify the satisfaction of condition 2 of Theorem \ref{thm4} here, which is much more stringent than the condition \eqref{LMI2}.
In contrast, by leveraging the known input matrix, {\color{black}our informativity condition is milder as indicated in Table \ref{tab}} and the approach for computing the state feedback control law for stabilization becomes more straightforward. Furthermore, we will show that, by appropriately selecting the cost function weights, the state feedback LQR solution of our ancillary system will approximate that of the original system while \cite{Persis2019} did not consider the relationship of the solutions to the LQR problem between the ancillary system and the original system.
\end{rem}}

Corresponding to Lemma \ref{lemStateInfo} and Theorem \ref{thm7}, we  immediately obtain the following two results, respectively:

{\color{black}\begin{lem}\label{stainfo2}    
	The data $\bar{\mathcal{D}}$ are informative for stabilization of \eqref{dyz2bf2}  by state feedback if and only if the matrix $\bar{X}_{-}$ has full row rank and there exists a right inverse $\bar{X}_{-}^{\dagger}$ of $\bar{X}_{-}$ and $T_\zeta \in \mathbb{R}^{m \times n_\zeta}$ such that $\bar{X}_{+}\bar{X}_{-}^{\dagger}+B_\zeta T_\zeta$ is stable.
\end{lem}}

and

\begin{thm}\label{thm14}
	The data $\bar{\mathcal{D}}$ are informative for stabilization of \eqref{dyz2bf2}  by state feedback if and only if there exists matrices $\bar{\Theta} \in \mathbb{R}^{\bar{T} \times n_\zeta}$ and $\bar{T}_p \in \mathbb{R}^{m \times n_\zeta}$ satisfies
	\begin{equation}\label{LMI2}
		\begin{aligned}
			\begin{bmatrix}
				\bar{X}_{-} \bar{\Theta} & \bar{X}_{+} \bar{\Theta}+B_\zeta \bar{T}_p\\
				(\bar{X}_{+}\bar{\Theta}+B_\zeta \bar{T}_p)^T & \bar{X}_{-}\bar{\Theta}
			\end{bmatrix} \succ 0.
		\end{aligned}
	\end{equation}
	Let $T_\zeta=\bar{T}_p (\bar{X}_{-} \bar{\Theta})^{-1}$. Then $K_\zeta$ satisfies $\bar{\Sigma}_{i/s} \subseteq \bar{\Sigma}_{K_\zeta}$ if and only if $K_\zeta=\bar{U}_{-} \bar{\Theta} (\bar{X}_{-}\bar{\Theta})^{-1}+T_\zeta$.
\end{thm}

{\color{black}
\begin{rem}
The robustness of the stabilization using the LMI-based approach has been studied in \cite{Persis2019}. Specifically, under sufficiently small perturbations of the collected data, the LMI-based approach can still provide a stabilizing feedback gain since the eigenvalues of $A + BK$ are continuous with respect to the gain $K$. This indicates that the  LMI-based approach has some intrinsic degree of robustness to small perturbations. 
\end{rem}
}

Corresponding to Theorem \ref{thmdeadbeat2}, we immediately obtain the following result:
\begin{thm} \label{thmnilp}
The data $\bar{\mathcal{D}}$ are informative for deadbeat control of \eqref{dyz2bf2} by state feedback if and only if the matrix $\bar{X}_{-}$ has full row rank and {\color{black}there exists a right inverse} $\bar{X}_{-}^{\dagger}$ of $\bar{X}_{-}$ and $T_\zeta \in \mathbb{R}^{m \times n_\zeta}$ such that $\bar{X}_{+}\bar{X}_{-}^{\dagger}+B_\zeta T_\zeta$ is nilpotent.
\end{thm}

\subsection{{\color{black}Relationship of the LQR Problem between System \eqref{lisys2} and the Ancillary System \eqref{dyz2bf2}}}\label{property}

The LQR problem for \eqref{lisys2} is to find a control law $u(t)=Kx(t)$ such that the cost $\sum_{t=0}^{\infty}(x^T(t)Qx(t)+u^T(t){R}u(t))$ is minimized for some  $Q \succeq 0, R \succ 0$.
On the other hand, the LQR problem for the ancillary system \eqref{dyz2bf2}
is to find a control law $u(t)=u_\zeta(t)={K}_\zeta \zeta(t)$ to minimize the cost $\sum_{t=0}^{\infty}(\zeta^T(t){Q_\zeta}\zeta(t)+u^T(t){R}u(t))$
for some  $Q_\zeta \succeq 0$   and $R \succ 0$. Clearly, if 	 $Q_\zeta= Q^T_\zeta \succeq 0$ is such that $(A_\zeta, \sqrt{Q_\zeta})$ is detectable and $(A_\zeta, B_\zeta)$ is stabilizable, then
	the following DARE
\begin{equation}\label{areliz}
	\begin{aligned}
		A_\zeta^T P_\zeta A_\zeta+Q_\zeta
		-A_\zeta^T P_\zeta B_\zeta(R+B_\zeta^T P_\zeta B_\zeta)^{-1}B_\zeta^T P_\zeta A_\zeta=P_\zeta
	\end{aligned}
\end{equation}	
admits a unique positive semidefinite solution $P^*_\zeta$ and the optimal gain is $K_\zeta^*=-(R+B_\zeta^T{P}_\zeta^*B_\zeta)^{-1}B_\zeta^T{P}_\zeta^*A_\zeta$.

\begin{rem}
In  \eqref{DARE}, $Q$ can be any symmetric positive semidefinite matrix such that $(A, \sqrt{Q})$ is detectable. On the other hand, in \eqref{areliz} ,  $Q_\zeta$ can be any symmetric positive semidefinite matrix such that $(A_\zeta, \sqrt{Q_\zeta})$ is detectable. The optimal gain $K^*_\zeta$ bears no relationship with the optimal gain $K^*$ for solving the LQR problem for system \eqref{lisys2}. Next, we will show that,  if  $Q = C^T Q_y C$ for some symmetric positive definite matrix $Q_y$ and  $Q_\zeta=M^TC^TQ_yCM$,  then $K^*_\zeta  =  K^* M$.
\end{rem}

Let us first state the following result.

\begin{lem}\label{lem3}
	Under Assumption 1, the pair $(A_\zeta,\sqrt{Q}_yCM)$ is detectable for any $Q_y \succ 0$.
\end{lem}
{\color{black}
\begin{Prf}
Since $Q_y \succ 0$, $\sqrt{{Q_y}}$ is an invertible matrix. Let $L_\zeta=\begin{bmatrix}
	{\bf0}\\I_p\otimes b
\end{bmatrix}\sqrt{{Q_y}}^{-1}$. Then, we have $A_\zeta-L_\zeta\sqrt{{Q_y}}CM=
I_{m+p} \otimes \mathcal{A}$. Since $\mathcal{A}$ is stable, $A_\zeta-L_\zeta\sqrt{{Q_y}}CM$ is stable. Thus, the pair $(A_\zeta, \sqrt{{Q_y}}CM)$ is detectable.
\end{Prf}}

\begin{rem}\label{rem10}
This result is parallel to  Lemma 7 of \cite{Lin2024}, which considered data-based control of  the continuous-time systems.

\end{rem}

\begin{rem}\label{rem11}
	If  $Q = C^T Q_y C$ for some $Q_y\succ 0$, then, under Assumption \ref{ass1}, the pair $(A, \sqrt{Q_y}C)$ is observable.
	Thus, the DARE \eqref{DARE} admits a unique positive definite solution ${P}^*$ and hence the LQR problem is solvable for $(A,B,C^TQ_yC,R)$ if and only if the pair $(A,B)$ is stabilizable.
	The optimal control law is given by  $K^*=-(R+B^T{P}^*B)^{-1}B^T{P}^* A$. On the other hand, by Lemma \ref{lem3},
	under Assumption \ref{ass1}, the pair   $(A_\zeta, \sqrt{{Q_y}}CM)$ is detectable.
\end{rem}

\begin{thm}\label{thm12}
	Under Assumption \ref{ass1}, and the assumption that the pair $(A, B)$ is stabilizable, let
	$Q = C^T Q_y C$ for some $Q_y\succ 0$ and $Q_\zeta=M^TC^TQ_yCM$. Then,
	the  algebraic Riccati equation \eqref{areliz}
admits a unique positive semidefinite solution ${P}_\zeta^*=M^TP^*M \succeq 0$, where $P^*$ is the unique positive definite solution to \eqref{DARE}, and the optimal gain $K_\zeta^*$ that solves the LQR problem  for $(A_\zeta,B_\zeta,M^TC^TQ_yCM,R)$ is such that $K^*_\zeta  =  K^* M$.
\end{thm}

\begin{proof}
First note that, by Remark \ref{rem11}, under the conditions of this theorem,  the DARE \eqref{DARE}  admits a unique positive definite solution ${P}^*$.
Under Assumption \ref{ass1}, the pair $(A_\zeta,\sqrt{Q}_yCM)$ is detectable by Lemma \ref{lem3}. Also, by Lemma \ref{lem2},  the stabilizability of the pair $(A, B)$ implies  that of the pair $(A_\zeta, B_\zeta)$.
Thus,  the DARE \eqref{areliz} admits a unique positive semidefinite solution ${P}_\zeta^*$.
	
	By \eqref{rela1}-\eqref{rela3}, we have
	\begin{equation}
		\begin{aligned}
			A_\zeta^T(M^TP^*M)A_\zeta=&M^TA^TP^*AM\\
			A_\zeta^T(M^TP^*M)B_\zeta=&M^TA^TP^*B\\
			B_\zeta^T(M^TP^*M)B_\zeta=&B^TP^*B
		\end{aligned}
	\end{equation}
	
We are now ready to show that $M^TP^*M$ is a solution to \eqref{areliz}. Substituting $P_\zeta=M^TP^*M$ into \eqref{areliz} gives
\begin{footnotesize}
		\begin{equation}
			\begin{aligned}
				&A_\zeta^T(M^TP^*M)A_\zeta+M^TQM-M^TP^*M\\
				&-A_\zeta^T(M^TP^*M)B_\zeta(R+B_\zeta^T(M^TP^*M)B_\zeta)^{-1}B_\zeta^T(M^TP^*M)A_\zeta\\
				=&M^TA^TP^*AM+M^TQM-M^TP^*M\\
				&-M^TA^TP^*B(R+B^TP^*B)^{-1}B^TP^*AM\\
				=&M^T(A^TP^*A+Q-P^*-A^TP^*B(R+B^TP^*B)^{-1}B^TP^*A)M\\
				=&\bf{0}
			\end{aligned}
		\end{equation}
	\end{footnotesize}
	
	Since \eqref{areliz} admits a unique positive semidefinite solution, $M^TP^*M$ is the unique positive semidefinite solution to \eqref{areliz}.

Moreover,  we have
\begin{equation}
	\begin{aligned}
		K_\zeta^*=&-(R+B_\zeta^T{P}_\zeta^*B_\zeta)^{-1}B_\zeta^T{P}_\zeta^*A_\zeta\\
		=&-(R+B^TP^*B)^{-1}B^TP^*AM\\
		=&K^*M
	\end{aligned}
\end{equation}

Since $M\zeta(t)$ converges to $x(t)$ exponentially, the control law $u^*_\zeta(t)=K_\zeta^* \zeta(t)=K^*M \zeta(t)$ converges to $u^*(t)=K^* x(t)$ exponentially.
Thus, the LQR solution for system \eqref{dyz2bf} will converge to the LQR solution for system \eqref{dyz2bf2}  exponentially.

\end{proof}


Now we consider the informative problem and control design problem for LQR of the ancillary system \eqref{dyz2bf2}.
For this purpose,
let
\begin{equation}
	\begin{aligned}
		\bar{\Sigma}_{K_\zeta}^{Q_\zeta, R}=\{ A_\zeta | K_\zeta \text{ is the optimal feedback gain for }\\ ( A_\zeta,B_\zeta,Q_\zeta,R)\},
	\end{aligned}
\end{equation}
that is, $\bar{\Sigma}_{K_\zeta}^{Q_\zeta, R}$
is the set of all system matrices $( A_\zeta, B_\zeta)$ such that $K_\zeta$ is the optimal feedback gain corresponding to $Q_\zeta$ and $R$.
The data $\bar{\mathcal{D}}$ are said to be informative for LQR of \eqref{dyz2bf2} if there exists a feedback gain $K_\zeta$ such that $\bar{\Sigma}_{i/s} \subseteq \bar{\Sigma}_{K_\zeta}^{Q_\zeta, R}$.

By Theorem \ref{thm8}, we have the following result on solving the informative problem for LQR.

\begin{thm}\label{thm15}
	Let $Q_\zeta= Q^T_\zeta \succeq 0$ be such that $(A_\zeta, \sqrt{Q_\zeta})$ is detectable and $R= R^T \succ 0$.   Then, under Assumption \ref{ass1}, the data $\bar{\mathcal{D}}$ are informative for LQR if and only if the matrix $\bar{X}_{-}$ has full row rank and there exists a right inverse $\bar{X}_{-}^{\dagger}$ of $\bar{X}_{-}$ and $T_\zeta \in \mathbb{R}^{m \times n_\zeta}$ such that $\bar{X}_{+}\bar{X}_{-}^{\dagger}+B_\zeta T_\zeta$ is stable.
\end{thm}

\begin{rem}\label{rem15}
By Remark \ref{rem11}, Theorem \ref{thm15} is particularly valid if  $Q_\zeta=M^TC^TQ_yCM$, where $Q_y \succ 0$.
\end{rem}

{\color{black}Now, applying} Theorem \ref{thm9} to the ancillary system \eqref{dyz2bf2} and noting  $\bar{Y}_- = C M \bar{X}_- = C X_-$ gives the following result.

\begin{thm}\label{thm16}
	Let $R \succ 0$ and $Q_\zeta=M^TC^TQ_yCM$, where $Q_y \succ 0$. Under Assumption \ref{ass1}, suppose the data $\bar{\mathcal{D}}$ are informative for LQR. Then, ${P}_\zeta^*$ in \eqref{areliz} is the unique solution to the following optimization problem

	\begin{pro} \label{opbar1}
		\begin{equation} \nonumber
			\begin{aligned}
				&\textup{max}~ \textup{tr}(P_\zeta)\\
				&\textup{subject to } P_\zeta \succeq 0\\
				&\begin{bmatrix}
					\bar{A}_\zeta^T P_\zeta \bar{A}_\zeta -\bar{X}_{-}^T {P}_\zeta \bar{X}_{-}+ \bar{Y}_{-}^T Q_y \bar{Y}_{-} & \bar{A}_\zeta^T P_\zeta B_\zeta \\
					B_\zeta ^T P_\zeta \bar{A}_\zeta & R+B_\zeta P_\zeta B_\zeta
				\end{bmatrix} \succeq 0
			\end{aligned}
		\end{equation}
		where $\bar{A}_\zeta =\bar{X}_{+}-B_\zeta \bar{U}_{-}$.
	\end{pro}

	Let ${P}_\zeta^*$ be the optimal solution for Problem \ref{opbar1}. Then, the optimal feedback gain $K_\zeta^*=\bar{U}_{-}\bar{X}_{-}^{\dagger}+T_\zeta^*$, where $\bar{X}_{-}^{\dagger}$ is any matrix such that $\bar{X}_{-} \bar{X}_{-}^{\dagger}=I_{n_\zeta}$, and  $T_\zeta^*=-(R+B_\zeta^T P_\zeta^* B_\zeta)^{-1} (B_\zeta^T P_\zeta^* \bar{X}_{+}+ R \bar{U}_{-}) \bar{X}_{-}^{\dagger}$.
\end{thm}

\begin{rem}
	Let $z(t)=\mbox{col} (\zeta(t),u(t))$, $n_z=mn+pn+m$, $\bar{z}(t)=\text{vecv}(z(t))$,
	$\Phi=[\bar{z}(T_0)-\bar{z}(T_0+1), \bar{z}(T_0+1)-\bar{z}(T_0+2), \cdots, \bar{z}(T_1)-\bar{z}(T_1+1)]$.
	{\color{black}
		In \cite{Rizvi2022}, in order to apply the Q-learning method, they need to fulfill the following data condition:
		\begin{equation}
			\begin{aligned}
				\textup{rank}(\Phi)=\frac{n_z(n_z+1)}{2}
			\end{aligned}
		\end{equation}
		This requirement is more stringent than the rank condition needed for system identification, which is
		\begin{equation}
			\begin{aligned}
				\textup{rank}\left(\bar{X}_{-} \right)=n_\zeta=mn+pn.
			\end{aligned}
		\end{equation}
		given that $B_\zeta$ is known.
		
		By Theorem \ref{thm15}, the rank condition required in our paper to solve the LQR problem is equivalent to the rank condition required for system identification. This equivalence implies that our approach avoids the more restrictive data requirements imposed by Q-learning, while still ensuring optimal control synthesis.}
\end{rem}


{\color{black}
\subsection{How to Choose the Designed Parameters in the Simulation}
The output-based stabilization, deadbeat control, and the LQR problem can all be addressed by the control law \eqref{outputc}, where $K_\zeta$ is obtained by solving an LMI, a pole placement, or an optimization problem, respectively. Thus, for stabilization and deadbeat control, the designed parameters are $\alpha_0, \cdots, \alpha_{n-1}$ of $\mathcal{A}$ and the required waiting time $T_0$. On the other hand,  the LQR problem additionally requires the weighting matrices $Q$ and $R$.  

For the stabilization problem, the designed parameters $\alpha_0, \cdots, \alpha_{n-1}$ of $\mathcal{A}$ are mainly  chosen by the requirement on the transient response of $\zeta$ dynamics.  For the deadbeat control, the observer must converge within a finite time, so the eigenvalues of $\mathcal{A}$ are placed at the origin. 
For the LQR  weight tuning, it is well known that  a larger $R$ incurs small control effort, while a larger $Q$
prioritizes faster state convergence.

The waiting time $T_0$  is also determined by the locations of eigenvalues of $\mathcal{A}$. 
Let $\epsilon$ 
be the tolerance range. The minimum waiting time can then be estimated as $T_0=\frac{\ln(\epsilon)}{\ln(a)}$, where $a$ is the largest magnitude of the eigenvalues of the matrix $\mathcal{A}$. 
}

\section{{\color{black}Conclusion and Future work}}\label{conclusion}
In this paper, we have first focused on developing data informativity analysis and control for an unknown linear discrete-time system with a known input matrix.
We have derived a set of much simpler conditions for data informativity for stabilization, deadbeat control, and the LQR for such systems, which in turn has led to {\color{black}more efficient methods} and flexibility
for computing state feedback control laws for stabilization, deadbeat control, and the LQR for such systems.
Further, we have shown that the problem of computing a dynamic output feedback control law for an unknown system can be converted to the problem of designing a state feedback control law for an ancillary system whose input
matrix is known. Therefore, the results of the first part of this paper can be directly used to {\color{black}compute a dynamic output feedback control law for stabilization and deadbeat control for unknown systems based on only the input and output data.}  Moreover, we have presented a dynamic output feedback control law which asymptotically approaches a state feedback LQR solution for the original unknown system.

{\color{black}
Since the mildest condition on the data given in this paper is equivalent to that required for system identification and the solvability of some LMIs. Consequently, a natural direction for future research is that: 1) further consider the control problems subject to noisy input and output data; 2) see if the informative conditions can be further relaxed; 3) combine our approach with the MPC framework to further address the state and input constraints;  and 4) apply our approach to some practical scenarios. }

\ifCLASSOPTIONcaptionsoff
  \newpage
\fi

\section{Appendix}
\subsection{Proof of Lemma \ref{lem2}}
\begin{Prf}
	If part:   Let $K$ be such that $A + B K$ is stable. {\color{black}To prove the pair $(A_\zeta, B_\zeta)$ is stabilizable, it suffices to show that,  under  $u = K_\zeta \zeta (t)$ with $K_\zeta = KM$,  the solution of the closed-loop system \eqref{dyz2bf} decays to zero exponentially.}
	In fact, applying $u(t)=K_\zeta \zeta (t) $  to \eqref{lisys2} gives
	\begin{equation}
		\begin{aligned}
			x(t+1)=&Ax(t)+BKM\zeta(t)\\
			=&Ax(t)+BK(x(t)+e_x(t))\\
			=&(A+BK)x(t)+BKe_x(t)
		\end{aligned}
	\end{equation}
	Since $e_x(t)$ tends to $\bf{0}$ exponentially, by Lemma 1 of \cite{Huang2016}, $x(t)$ converges to $\bf{0}$ exponentially since $A+BK$ is stable.
	
	Next, applying $u(t)=KM\zeta(t)$ to \eqref{dyz2bf} and using $M\zeta(t)=x(t)+e_x(t)$ gives
	\begin{equation}
		\begin{aligned}
			\zeta(t+1)=&(I_{m+p}\otimes \mathcal{A})\zeta(t)+\begin{bmatrix}
				{\bf0}\\I_p\otimes b
			\end{bmatrix}CM\zeta(t) \\
			&+B_\zeta K M\zeta(t)  - \begin{bmatrix}
				{\bf0} \\ I_p\otimes b
			\end{bmatrix} C e_x(t)  \\
			=&(I_{m+p}\otimes \mathcal{A})\zeta(t)+ \left  (\begin{bmatrix}
				{\bf0}\\I_p\otimes b
			\end{bmatrix}C +B_\zeta K \right ) x(t) \\
			&+  B_\zeta K   e_x(t)  	
		\end{aligned}
	\end{equation}
	Since both $x(t)$ and $e_x(t)$ converge to $\bf{0}$ exponentially and $(I_{m+p}\otimes \mathcal{A})$ is stable, the solution of $\zeta(t)$ goes to $\bf{0}$ exponentially by Lemma 1 of \cite{Huang2016} again. Thus,  we conclude that the pair $(A_\zeta, B_\zeta)$ must be stabilizable.

Only if part:
We will prove that the pair $(A,B)$ is stabilizable by contradition. Assume that the pair $(A,B)$ is not stabilizable which implies  $\exists |\lambda_1| \geq 1$, $\exists s_1 \neq {\bf 0} \in \mathbb{R}^n$, such that 
\begin{equation}\label{contradict1}
\begin{aligned}
s_1^T\begin{bmatrix}\lambda_1 I_n-A & B \end{bmatrix}=\bf{0}. 
\end{aligned}
\end{equation}
{\color{black}
Note that the pair $(A, [B~L])$ is stabilizable since $A+\begin{bmatrix} B & L\end{bmatrix}\begin{bmatrix} \bf{0} \\ -C\end{bmatrix}=A-LC$ is stable by Assumption \ref{ass1}. By PBH test, 
\begin{equation}\label{contradict2}
\begin{aligned}
s_1^T\begin{bmatrix}\lambda_1 I_n-A & B &L \end{bmatrix}\neq \bf{0}.
\end{aligned}
\end{equation}}\eqref{contradict1} and \eqref{contradict2} imply $s_1^TL \neq \bf{0}$. Thus, $s_1^TM \neq \bf{0}$ by the definition of $M$. With the pair $(A_\zeta, B_\zeta)$ stabilizable, $s_1^TM \begin{bmatrix}\lambda_1 I_{n_\zeta}-A_\zeta & B_\zeta \end{bmatrix}\neq \bf{0}$.
	\begin{equation}
		\begin{aligned}
			s_1^TM \begin{bmatrix}\lambda_1 I_{n_\zeta}-A_\zeta & B_\zeta \end{bmatrix}=s_1^T \begin{bmatrix}\lambda_1 I_n-A &  B\end{bmatrix}\begin{bmatrix}M & \bf{0}\\ \bf{0} & I\end{bmatrix} \neq \bf{0}
		\end{aligned}
	\end{equation}
	which contradicts the fact that $s_1^T\begin{bmatrix}\lambda_1 I_n-A & B \end{bmatrix}=\bf{0}$.
	Thus, the pair $(A,B)$ is stabilizable.
\end{Prf}

\subsection{Proof of Theorem \ref{thm10}}
\begin{Prf}
	If part: Let  $L$ be such that  $A-LC$ is stable, $K$ be such that $A+B K$ is stable, and $K_\zeta = K M$. {\color{black}Then, applying the control law $u(t)=  K_\zeta \zeta(t) $ to \eqref{lisys2} and \eqref{dyz2bf} gives the following closed-loop system:
		\begin{equation} \label{closedloop system}
			\begin{aligned}
				\begin{bmatrix}
					x(t+1)\\
					\zeta(t+1)
				\end{bmatrix}
				=\begin{bmatrix}
					A & BK_\zeta\\
					\begin{bmatrix}
						{\bf0}\\I_p\otimes b
					\end{bmatrix} C& (I_{m+p}\otimes \mathcal{A})+ B_\zeta K_\zeta
				\end{bmatrix}
				\begin{bmatrix}
					x(t)\\
					\zeta(t)
				\end{bmatrix}
			\end{aligned}
		\end{equation}
		Consider the nonsingular transformation $H=\begin{bmatrix}
			I_n & -M\\
			\bf{0} & I_{n_\zeta}
		\end{bmatrix}$. 
		\begin{equation} \label{nonsingularH}
			\begin{aligned}
				&H \begin{bmatrix}
					A & BK_\zeta\\
					\begin{bmatrix}
						{\bf0}\\I_p\otimes b
					\end{bmatrix} C& (I_{m+p}\otimes \mathcal{A})+ B_\zeta K_\zeta
				\end{bmatrix} H^{-1}\\
				=&\begin{bmatrix}
					A-LC &  \bf{0} \\
					\begin{bmatrix}
						{\bf0}\\I_p\otimes b
					\end{bmatrix} C& A_\zeta+ B_\zeta K_\zeta
				\end{bmatrix} 
			\end{aligned}
	\end{equation}}Since $A_\zeta+B_\zeta K_\zeta$ is stable by Lemma \ref{lem2},
	the closed-loop system \eqref{closedloop system}
	is exponentially stable.
	
	Only if part: If the closed-loop system \eqref{closedloop system} is stable, by \eqref{nonsingularH}, $(A_\zeta+B_\zeta K_\zeta) $ must be stable. By Lemma \ref{lem2},  $(A, B)$ must be stabizable.
\end{Prf}

\subsection{Proof of Theorem \ref{thmdeadbeat3}}
\begin{Prf}
	If part: Under Assumption 1, let $L$ be such that all the eigenvalues of $A-LC$ equal zero. Then, $(A-LC)^n=\bf{0}$ and,  hence, by \eqref{ex}, for any $e_x(0) \in \mathbb{R}^n$,  $e_x(t)=\bf{0}$ for all $t \geq n$.

	Let $K_\zeta = K M$. Using the relation  $x(t)=M\zeta(t)-e_x(t)$, we have, for all $t \geq n$,   $u(t)= K_\zeta \zeta(t) =Kx(t)$.
	Applying the output feedback  control law \eqref{outputc}   to \eqref{lisys2} renders  the  closed-loop system the following form, for $t \geq n$,
	
	\begin{equation*}
		\begin{aligned}
			x(t+1)=&(A+BK)x(t)  \\
			\zeta(t+1)=& (I_{m+p}\otimes \mathcal{A})
			\zeta(t) \\
			&+\begin{bmatrix}
				I_m\otimes b\\ {\bf0}
			\end{bmatrix}Kx(t)+\begin{bmatrix}
				{\bf0}\\I_p\otimes b
			\end{bmatrix} y (t)
		\end{aligned}
	\end{equation*}
	Since $K$ is such that $(A+BK)^n=\bf{0}$, we have $x(t)=\bf{0}$ for all $t \geq 2n$. As a result, for $ t \geq 2n$,
	the  $\zeta(t)$ dynamics reduces to $\zeta(t+1)=(I_{m+p}\otimes \mathcal{A})
	\zeta(t)$. Since the eigenvalues of $ \mathcal{A}$ are the same as those of $A-LC$, $\mathcal{A}^n=\bf 0$. Thus, all the solutions of  the closed-loop system are driven to zero after $t \geq 3n$.
	
	Only if part: 
{\color{black}We will prove that there exists a $K \in \mathbb{R}^{m \times n}$ such that $(A+BK)^n = \bf 0$ by contradiction.} Suppose $\forall K \in \mathbb{R}^{m \times n}$, $(A+BK)^n \neq \bf 0$,  which implies the existence of  $\lambda_1 \neq 0$, $s_1  \neq {\bf 0} \in \mathbb{R}^n$ such that
\begin{equation}\label{contradict3}
\begin{aligned}
s_1^T\begin{bmatrix}\lambda_1 I_n-A & B \end{bmatrix}=\bf{0}. 
\end{aligned}
\end{equation}

If all the solutions of the closed-loop system \eqref{closedloop system} are driven to zero in a finite time, then, by \eqref{nonsingularH}, all the eigenvalues of $A_\zeta+B_\zeta K_\zeta$ and $A-LC$ must be zero. Thus, by PBH test,
\begin{equation}\label{contradict4}
\begin{aligned}
s_1^T\begin{bmatrix}\lambda_1 I_n-A  &L \end{bmatrix}\neq \bf{0}
\end{aligned}
\end{equation}
By  \eqref{contradict3} and \eqref{contradict4}, $s_1^T L \neq \bf 0$, which implies $s_1^T M \neq \bf 0$ by the definition of $M$. Since all the eigenvalues of $A_\zeta+B_\zeta K_\zeta$  must be zero, by PBH test, 
\begin{equation*}
	\begin{aligned}
		s_1^TM \begin{bmatrix}\lambda_1 I_{n_\zeta}-A_\zeta & B_\zeta \end{bmatrix}=s_1^T \begin{bmatrix}\lambda_1 I_n-A &  B\end{bmatrix}\begin{bmatrix}M & \bf{0}\\ \bf{0} & I\end{bmatrix} \neq \bf{0}
	\end{aligned}
\end{equation*}
which contradicts the fact that  $s_1^T\begin{bmatrix}\lambda_1 I_n-A & B \end{bmatrix}=\bf{0}$.
The proof is thus completed.
\end{Prf}

\end{document}